\documentclass[11pt]{amsart}
\usepackage{amscd,amssymb}
\usepackage{amsthm,amsmath,amssymb}
\usepackage[matrix,arrow]{xy}

\newtheorem{theorem}[equation]{Theorem}

\newtheorem{lemma}[equation]{Lemma}
\newtheorem{corollary}[equation]{Corollary}

\theoremstyle{definition}
\newtheorem{example}[equation]{Example}
\newtheorem{definition}[equation]{Definition}

\theoremstyle{remark}
\newtheorem{remark}[equation]{Remark}

\makeatletter\@addtoreset{equation}{section} \makeatother

\author{Ivan Cheltsov}

\title{On singular cubic surfaces}

\address{\begin{tabbing}
\hspace*{28 em}\=\kill
School of Mathematics\\
University of Edinburgh\\
Edinburgh EH9 3JZ, UK\\
\\
\texttt{I.Cheltsov@ed.ac.uk}
\end{tabbing}}

\thanks{The author would like to thank I.\,Dolgachev, J.\,Koll\'ar,
J.\,Park, V.\,Sho\-ku\-rov, V.\,Iskovskikh for useful comments.}

\begin{document}

\begin{abstract}
We study global log canonical thresholds of cubic surfaces with
canonical singularities, and we prove the existence of a
K\"ahler--Einstein metric on two singular cubic~surfaces.
\end{abstract}

\maketitle

\section{Introduction.}
\label{section:intro}

Let $X$ be a Fano variety\footnote{We assume that all varieties
are projective, normal, and defined over $\mathbb{C}$.} with log
terminal singularities, and $G$ be a finite subgroup in
$\mathrm{Aut}(X)$.

\begin{definition}
\label{definition:threshold} Global $G$-invariant log canonical
threshold of the variety $X$ is the number
$$
\mathrm{lct}\big(X,G\big)=\mathrm{sup}\left\{\lambda\in\mathbb{Q}\ \left|%
\aligned
&\text{the log pair}\ \left(X, \lambda D\right)\ \text{has log canonical singularities}\\
&\text{for every $G$-invariant effective $\mathbb{Q}$-divisor}\ D\equiv -K_{X}\\
\endaligned\right.\right\}.%
$$
\end{definition}

We put $\mathrm{lct}(X)=\mathrm{lct}(X,G)$ in the case when $G$ is
a trivial group.

\begin{example}
\label{example:Cheltsov-Park} Let $X$ be a smooth hypersurface in
$\mathbb{P}^{n}$ of degree $n$. Then $\mathrm{lct}(X)\geqslant
1-1/n$~by~\cite{Ch01b}.
\end{example}

\begin{example}
\label{example:Klein} The simple group
$\mathrm{PGL}(2,\mathbb{F}_{7})$ is a group of automorphisms of
the quartic curve
$$
x^{3}y+y^{3}z+z^{3}x=0\subseteq\mathbb{P}^{2}\cong\mathrm{Proj}\Big(\mathbb{C}[x,y,z]\Big),
$$
which induces an embedding
$\mathrm{PGL}(2,\mathbb{F}_{7})\subseteq\mathrm{Aut}(\mathbb{P}^{2})$.
Then $\mathrm{lct}(\mathbb{P}^{2},
\mathrm{PGL}(2,\mathbb{F}_{7}))=4/3$ by \cite{Ch07b}.
\end{example}

The number $\mathrm{lct}(X, G)$ plays an important role in
birational geometry (see Section~\ref{section:fibers}).

\begin{example}
\label{example:CPR} Let $X$ be a general quasismooth hypersurface
in $\mathbb{P}(1,a_{1},\ldots,a_{4})$ of
degree~$\sum_{i=1}^{4}a_{i}$ with terminal singularities such that
$-K_{X}^{3}\leqslant 1$. Then $\mathrm{lct}(X)=1$ by \cite{Ch07a},
which implies that
$$
\mathrm{Bir}\Big(\underbrace{X\times\cdots\times X}_{m\ \text{times}}\Big)=\Big<\prod_{i=1}^{m}\mathrm{Bir}\big(X\big),\ \mathrm{Aut}\Big(\underbrace{X\times\cdots\times X}_{m\ \text{times}}\Big)\Big>,%
$$%
and the variety $X\times\cdots\times X$ is non-rational (see
\cite{CPR}, \cite{Pu04d}, \cite{Ch07a}).
\end{example}

The number $\mathrm{lct}(X, G)$ plays an important role in
K\"ahler geometry.

\begin{example}
\label{example:KE} Suppose that $X$ has at most quotient
singularities, and the inequality
$$
\mathrm{lct}\big(X,
G\big)>\frac{\mathrm{dim}\big(X\big)}{\mathrm{dim}\big(X\big)+1}
$$
holds. Then $X$ has a K\"ahler--Einstein metric (see
\cite{DeKo01}).
\end{example}

Let $S$ be a del Pezzo surface with canonical singularities. Put
$\Sigma=\mathrm{Sing}(S)$.

\begin{remark}
\label{remark:KE-exists} It follows from \cite{Ti90},
\cite{DiTi92}, \cite{MaMu93}, \cite{Je97}, \cite{GhKo05},
\cite{Ch07b} that
\begin{itemize}
\item the surface $S$ has a K\"ahler--Einstein metric in the following cases:%
\begin{itemize}
\item when $\Sigma=\varnothing$, $S\not\cong\mathbb{F}_{1}$ and $K_{S}^{2}\ne 7$;%
\item when $S$ is a complete intersection
$$
\sum_{i=0}^{4}x_{i}^{2}=\sum_{i=0}^{4}\lambda_{i}x_{i}^{2}=0\subseteq\mathbb{P}^{4}\cong\mathrm{Proj}\Big(\mathbb{C}[x_{0},\ldots,x_{4}]\Big),
$$
and $\Sigma$ consists of points of type $\mathbb{A}_{1}$, where $\lambda_{i}\in\mathbb{C}$;%
\item when $K_{S}^{2}=2$, and $\Sigma$ consists of points of points of types  $\mathbb{A}_{1}$ and $\mathbb{A}_{2}$;%
\item when $K_{S}^{2}=1$, and $\Sigma$ consists of points of type $\mathbb{A}_{1}$;%
\end{itemize}
\item the surface $S$ does not have a K\"ahler--Einstein metric in the following~cases:%
\begin{itemize}
\item when $\Sigma=\varnothing$, and either $S\cong\mathbb{F}_{1}$ or $K_{S}^{2}=7$;%
\item when $\Sigma$ contains a point that is not of type  $\mathbb{A}_{1}$, and $K_{S}^{2}=4$;%
\item when $\Sigma$ contains a point that is not of type  $\mathbb{A}_{1}$ or $\mathbb{A}_{2}$, and $K_{S}^{2}=3$.%
\end{itemize}
\end{itemize}
\end{remark}

All possible values of $\mathrm{lct}(S)$ are found in \cite{Ch07b}
in the case when $\Sigma=\varnothing$.

\begin{example}
\label{example:smooth-del-Pezzo-surfaces} Suppose that $S$ is a
cubic surface in $\mathbb{P}^{3}$ and $\Sigma=\varnothing$. Then
$$
\mathrm{lct}\big(S\big)=\left\{%
\aligned
&2/3\ \text{when}\ S\ \text{has an Eckardt point},\\%
&3/4\ \text{when}\ S\ \text{does not have Eckardt points}.\\%
\endaligned\right.%
$$
\end{example}

We prove the following result in Sections~\ref{section:A1}.

\begin{theorem}
\label{theorem:main} Suppose that $S$ is a cubic surface in
$\mathbb{P}^{3}$ and $\Sigma\ne\varnothing$. Then
$$
\mathrm{lct}\big(S\big)=\left\{%
\aligned
&2/3\ \text{when}\ \Sigma=\big\{\mathbb{A}_{1}\big\},\\%
&1/3\ \text{when}\ \Sigma\supseteq\big\{\mathbb{A}_{4}\big\},\\%
&1/3\ \text{when}\ \Sigma=\big\{\mathbb{D}_{4}\big\},\\%
&1/3\ \text{when}\ \Sigma\supseteq\big\{\mathbb{A}_{2},\mathbb{A}_{2}\big\},\\%
&1/4\ \text{when}\ \Sigma\supseteq\big\{\mathbb{A}_{5}\big\},\\%
&1/4\ \text{when}\ \Sigma=\big\{\mathbb{D}_{5}\big\},\\
&1/6\ \text{when}\ \Sigma=\big\{\mathbb{E}_{6}\big\},\\%
&1/2\ \text{in other cases}.\\%
\endaligned\right.%
$$
\end{theorem}

The group $\mathrm{S}_{4}$ naturally acts on the cubic surface
$\grave{S}\subset\mathbb{P}^{3}$ that is given by the equation
\begin{equation}
\label{equation:Cayley-cubic}
xyz+xyt+xzt+yzt=0\subseteq\mathbb{P}^{3}\cong\mathrm{Proj}\Big(\mathbb{C}[x,y,z,t]\Big),
\end{equation}
the group $\mathrm{S}_{3}\times\mathbb{Z}_{3}$ naturally acts on
the cubic surface $\acute{S}\subset\mathbb{P}^{3}$ that is given
by the equation
\begin{equation}
\label{equation:cubic-A2-A2-A2}
xyz=t^{3}\subseteq\mathbb{P}^{3}\cong\mathrm{Proj}\Big(\mathbb{C}[x,y,z,t]\Big),
\end{equation}
and
$\mathrm{lct}(\grave{S},\mathrm{S}_{4})=\mathrm{lct}(\acute{S},\mathrm{S}_{3}\times\mathbb{Z}_{3})=1$
(see Section~\ref{section:invariants}). But both surfaces
$\grave{S}$ and $\acute{S}$ are singular.

\begin{corollary}
\label{corollary:KE-singular-cubics} The surfaces~$\grave{S}$ and
$\acute{S}$ have K\"ahler--Einstein metrics.
\end{corollary}

It is very likely that the method in \cite{Ti90} can be applied to
prove the existence of a K\"ahler--Ein-\break stein metric on
every singular cubic surface having only singular points of
type~$\mathbb{A}_{1}$~and~$\mathbb{A}_{2}$.

\section{Basic tools.}
\label{section:tools}

Let $S$ be a surface with canonical singularities, and $D$ be an
effective $\mathbb{Q}$-divisor on it.

\begin{remark}
\label{remark:convexity} Let $B$ be an effective
$\mathbb{Q}$-divisor on $S$ such that $(S,B)$ is log canonical.
Then
$$
\left(S,\ \frac{1}{1-\alpha}\Big(D-\alpha B\Big)\right)
$$
is not log canonical if $(S, D)$ is not log canonical, where
$\alpha\in\mathbb{Q}$ such that $0\leqslant\alpha<1$.
\end{remark}

Let $\mathrm{LCS}(S,D)\subset S$ be a subset such that
$P\in\mathrm{LCS}(S,D)$ if and only if $(S,D)$ is not log~terminal
at the point $P$. The set $\mathrm{LCS}(S,D)$ is called the locus
of log canonical singularities.

\begin{remark}
\label{remark:connectedness} The set $\mathrm{LCS}(S,D)$ is
connected if $-(K_{S}+D)$ is ample (see Theorem~17.4 in
\cite{Ko91}).
\end{remark}

Let $P$ be a point of the surface $S$ such that $(S,D)$ is not log
canonical at the point $P$.

\begin{remark}
\label{remark:smooth-points} Suppose that $S$ is smooth at $P$.
Then $\mathrm{mult}_{P}(D)>1$.
\end{remark}

Let $C$ be an irreducible curve on the surface $S$. Put
$$
D=mC+\Omega,
$$
where $m\in\mathbb{Q}$ such that $m\geqslant 0$, and $\Omega$ is
an effective $\mathbb{Q}$-divisor such that
$C\not\subseteq\mathrm{Supp}(\Omega)$.

\begin{remark}
\label{remark:curves} Suppose that $C\subseteq\mathrm{LCS}(S,D)$.
Then $m\geqslant 1$.
\end{remark}

Suppose that $C$ is smooth at $P$, the inequality $m\leqslant 1$
holds and $P\in C$.

\begin{remark}
\label{remark:adjunction} Suppose that $S$ is smooth at $P$. Then
it follows from Theorem~17.6 in \cite{Ko91} that
$$
C\cdot\Omega\geqslant\mathrm{mult}_{P}\Big(\Omega\big\vert_{C}\Big)>1.
$$
\end{remark}

Let $\pi\colon \bar{S}\to S$ be a birational morphism, and
$\bar{D}$ is a proper transform of $D$ via $\pi$. Then
$$
K_{\bar{S}}+\bar{D}+\sum_{i=1}^{r}a_{i}E_{i}\equiv\pi^{*}\big(K_{S}+D\big),
$$
where $E_{i}$ is a $\pi$-exceptional curve, and $a_{i}$ is a
rational number.

\begin{remark}
\label{remark:log-pull-back} The log pair $(S,D)$ is log canonical
if and only if $(\bar{S},\bar{D}+\sum_{i=1}^{r}a_{i}E_{i})$ is log
canonical.
\end{remark}

Suppose that $r=1$, $\pi(E_{1})=P$, and $P$ is a singular point of
the surface $S$ of type $\mathbb{A}_{n}$.

\begin{remark}
\label{remark:A1} Suppose that  $n=1$, and $\bar{S}$ is smooth
along $E_{1}$. Then~$a_{1}>1/2$.
\end{remark}

Suppose that  $n>1$, and $E_{1}\cap\mathrm{Sing}(\bar{S})$
consists of one singular point of type $\mathbb{A}_{n-1}$.

\begin{remark}
\label{remark:An} It follows from Theorem~17.6 in \cite{Ko91} that
$a_{1}>1/(n+1)$.
\end{remark}

Most of the described results are valid in much more general
settings (see \cite{Ko91}).

\section{Main result.}
\label{section:A1}

Let us use the assumptions and notations of
Theorem~\ref{theorem:main}. Put
$$
\omega=\left\{%
\aligned
&2/3\ \text{when}\ \Sigma=\big\{\mathbb{A}_{1}\big\},\\%
&1/3\ \text{when}\ \Sigma\supseteq\big\{\mathbb{A}_{4}\big\},\\%
&1/3\ \text{when}\ \Sigma=\big\{\mathbb{D}_{4}\big\},\\%
&1/3\ \text{when}\ \Sigma\supseteq\big\{\mathbb{A}_{2},\mathbb{A}_{2}\big\},\\%
&1/4\ \text{when}\ \Sigma\supseteq\big\{\mathbb{A}_{5}\big\},\\%
&1/4\ \text{when}\ \Sigma=\big\{\mathbb{D}_{5}\big\},\\
&1/6\ \text{when}\ \Sigma=\big\{\mathbb{E}_{6}\big\},\\%
&1/2\ \text{in other cases}.\\%
\endaligned\right.%
$$

\begin{remark}
\label{remark:Joonyeong-Won} It follows from \cite{Wo06} that
$$
w=\mathrm{sup}\Big\{\mu\in\mathbb{Q}\ \big\vert\ \text{the log pair}\ \big(S,\mu D\big)\ \text{is log canonical for every}\ D\in\big|-K_{X}\big|\Big\}.%
$$
\end{remark}

Take $\lambda\in\mathbb{Q}$ such that $\lambda<\omega$. Let $D$ be
any effective $\mathbb{Q}$-divisor on $S$ such that~$D\equiv
-K_{S}$.

\begin{lemma}
\label{lemma:smooth-points-1-3} Suppose that $\lambda<1/3$. Then
$\mathrm{LCS}(S,\lambda D)\subseteq\Sigma$.
\end{lemma}

\begin{proof}
Suppose that $(S,\lambda D)$ is not log terminal at a smooth point
$P\in S$. Then
$$
3=-K_{S}\cdot D\geqslant\mathrm{mult}_{P}\big(D\big)>1/\lambda>3,%
$$
which is a contradiction.
\end{proof}

\begin{lemma}
\label{lemma:smooth-points-LCS} Suppose that
$|\mathrm{LCS}(S,\lambda D)|<+\infty$. Then
$\mathrm{LCS}(S,\lambda D)\subseteq\Sigma$.
\end{lemma}

\begin{proof}
The necessary assertion follows from \cite{Ch01b} or \cite{Ch07b}.
\end{proof}

Let $O$ be the worst singular point of the surface $S$, and
$\alpha\colon\bar{S}\to S$ be a partial resolution of
sin\-gu\-la\-ri\-ties that contracts smooth rational curves
$E_{1},\ldots,E_{k}$ to the point $O$ such that
$$
\bar{S}\setminus\Big(\cup_{i=1}^{k}E_{i}\Big)\cong S\setminus O,%
$$
the surface $\bar{S}$ is smooth along $\cup_{i=1}^{k}E_{i}$, and
$E_{i}^{2}=-2$ for every $i=1,\ldots,k$. Then
$$
\bar{D}\equiv\alpha^{*}\big(D\big)-\sum_{i=1}^{k}a_{i}E_{i},
$$
where $\bar{D}$ is the proper transform of $D$ on the surface
$\bar{S}$, and $a_{i}\in\mathbb{Q}$. Let $L_{1},\ldots,L_{r}$ be
lines on the surface $S$ such that $O\in L_{i}$, and $\bar{L}_{i}$
be the~proper transform of $L_{i}$ on the surface $\bar{S}$.

\begin{lemma}
\label{lemma:dP3-A1-single} Suppose that
$\Sigma=\{\mathbb{A}_{1}\}$. Then $\mathrm{lct}(S)=2/3$.
\end{lemma}

\begin{proof}
There is a conic $C_{i}\subset S$ such that the singularities of
the log pair $(S,\frac{2}{3}(L_{i}+C_{i}))$ are log canonical and
not log terminal. So, we may assume that $(S,\lambda D)$ is not
log canonical.

Suppose that there is an irreducible curve $Z\subset S$ such that
$D=\mu Z+\Omega$, where $\mu$ is a rational number such that
$\mu\geqslant 1/\lambda$, and $\Omega$ is an~effective
$\mathbb{Q}$-divisor such that $Z\not\subset\mathrm{Supp}(S)$.
Then
$$
3=-K_{S}\cdot
D=\mu\mathrm{deg}\big(Z\big)-K_{S}\cdot\Omega\geqslant\mu\mathrm{deg}\big(Z\big)>3\mathrm{deg}\big(Z\big)/2,
$$
which implies that $Z$ is a line. Let $C$ be a general conic on
$S$ such that $-K_{S}\sim Z+C$. Then
$$
2=C\cdot D=\mu C\cdot Z+C\cdot\Omega\geqslant \mu C\cdot Z\geqslant\frac{3}{2}\mu,%
$$
which is a contradiction. Then $\mathrm{LCS}(S,\lambda D)=O$ by
Lemma~\ref{lemma:smooth-points-LCS}. We have
$3-2a_{1}=\bar{H}\cdot\bar{D}\geqslant 0$, where $\bar{H}$ is a
general curve in $|-K_{\bar{S}}-E_{1}|$. Thus, it follows from
the~equivalence
$$
K_{\bar{S}}+\lambda\bar{D}+\lambda a_{1} E_{1}\equiv\alpha^{*}\big(K_{S}+\lambda D\big)%
$$
that there is a point $Q\in E_{1}$ such that
$(\bar{S},\lambda\bar{D}+\lambda a_{1} E_{1})$ is not log
canonical at the point $Q$.

Suppose that $Q\not\in\cup_{i=1}^{6}\bar{L}_{i}$. Let $\pi\colon
\bar{S}\to\mathbb{P}^{2}$ be a contraction of the curves
$\bar{L}_{1},\ldots,\bar{L}_{6}$. Then
$$
\pi\big(\bar{D}+a_{1} E_{1}\big)\equiv \pi\big(-K_{\bar{S}}\big)\equiv -K_{\mathbb{P}^{2}},%
$$
and $\pi$ is an isomorphism in a neighborhood of $Q$. Let $L$ be a
general line on $\mathbb{P}^{2}$. Then the locus
$$
\mathrm{LCS}\Big(\mathbb{P}^{2},\ L+\pi\big(\lambda\bar{D}+\lambda a_{1} E_{1}\big)\Big)%
$$
is not not connected, which is impossible by
Remark~\ref{remark:connectedness}.

Therefore, we may assume that $Q\in\bar{L}_{1}$. Put
$D=aL_{1}+\Upsilon$, where $a$ is a non-negative rational number,
and $\Upsilon$ is an~effective $\mathbb{Q}$-divisor, whose support
does not contain the line $L_{1}$. Then
$$
\bar{\Upsilon}\equiv\alpha^{*}\big(\Upsilon\big)-\epsilon E_{1},
$$
where $\epsilon=a_{1}-a/2$, and $\bar{\Upsilon}$ is the proper
transforms of the divisor $\Upsilon$ on the~surface $\bar{S}$.

The log pair $(\bar{S},\lambda
a\bar{L}_{1}+\lambda\bar{\Upsilon}+\lambda(a/2+\epsilon)E_{1})$ is
not log canonical at $Q$. Then
$$
1+a/2-\epsilon=\bar{L}_{1}\cdot\bar{\Upsilon}>1/\lambda-a/2-\epsilon
$$
by Remark~\ref{remark:adjunction}, because $\lambda a\leqslant 1$.
Hence, we have $a>1/2$.

We may assume that  $\mathrm{Supp}(D)$ does not contain the conic
$C_{1}$ due to Remark~\ref{remark:convexity}. Then
$$
2-3a/2-\epsilon=\bar{C}_{1}\cdot\bar{\Upsilon}\geqslant\mathrm{mult}_{Q}\big(\bar{\Upsilon}\big)>1/\lambda-a/2-\epsilon,
$$
where $\bar{C}_{1}$ be the proper transforms of $C_{1}$ on
the~surface $\bar{S}$. Hence, we see that $a<1/2$.
\end{proof}

\begin{lemma}
\label{lemma:dP3-A1-many} Suppose that
$\Sigma=\{\mathbb{A}_{1},\ldots,\mathbb{A}_{1}\}$ and
$|\Sigma|\geqslant 2$. Then $\mathrm{lct}(S)=1/2$.
\end{lemma}

\begin{proof}
Let $P$ be a point in $\Sigma$ such that $P\ne O$. We may assume
that $P\in L_{1}$. Then
$$
2L_{1}+L^{\prime}\sim -K_{S}
$$
for some line $L^{\prime}\subset S$. Hence, we may assume that
$(S,\lambda D)$ is not log canonical.

Suppose that there is an irreducible curve $Z$ on the surface $S$
such that
$$
D=\mu Z+\Omega,
$$
where $\mu$ is a rational number such that $\mu\geqslant
1/\lambda$, and $\Omega$ is an~effective $\mathbb{Q}$-divisor,
whose support does not contain the curve $Z$. Then $Z$ is a line
(see the proof of Lemma~\ref{lemma:dP3-A1-single}). We have
$$
2=C\cdot D=\mu C\cdot Z+C\cdot\Omega\geqslant \mu C\cdot Z\geqslant\mu\geqslant 1/\lambda>2,%
$$
where $C$ is a general conic on $S$ that intersects $Z$ in two
points.

We see that $\mathrm{LCS}(S,\lambda D)=O$ and $a_{1}>1$ (see
Lemma~\ref{lemma:smooth-points-LCS}~and~Remarks~\ref{remark:connectedness}~and~\ref{remark:A1}).

Arguing as in the proof of Lemma~\ref{lemma:dP3-A1-single}, we see
that there is a point $Q\in E$ such that the~singularities of
the~log pair $(\bar{S},\lambda\bar{D}+\lambda a_{1} E_{1})$ are
not log~canonical at the point $Q$.

Suppose that $Q\in\bar{L}_{1}$. Let $a$ be a non-negative rational
number such that
$$
D=aL_{1}+\Upsilon,
$$
where $\Upsilon$ is an~effective $\mathbb{Q}$-divisor, whose
support does not contain the line $L_{1}$. Then
$$
\bar{\Upsilon}\equiv\alpha^{*}\big(\Upsilon\big)-\epsilon E_{1},
$$
where $\bar{\Upsilon}$ is the proper transforms of $\Upsilon$ on
the~surface $\bar{S}$, and $\epsilon=a_{1}-a/2$. The log pair
$$
\Big(\bar{S},\ \lambda
a\bar{L}_{1}+\lambda\bar{\Upsilon}+\lambda(a/2+\epsilon)E_{1}\Big)
$$
is not log canonical at the point $Q$. We have
$\bar{L}_{1}^{2}=-1/2$. Then
$$
1-\epsilon=\bar{L}_{1}\cdot\bar{\Upsilon}>1/\lambda-a/2-\epsilon
$$
by Remark~\ref{remark:adjunction}. We have $a>1/\lambda$, which is
impossible. Hence, we see that $Q\not\in\bar{L}_{1}$.

There is a unique reduced conic $Z\subset S$ such that $O\in Z\ni
P$ and $Q\in\bar{Z}$, where $\bar{Z}$ is~the~proper transform of
the conic $Z$ on the surface $\bar{S}$. Then
$L_{1}\not\subseteq\mathrm{Supp}(Z)$, because
$Q\not\in\bar{L}_{1}$.

Suppose that $Z$ is irreducible. Put $D=eZ+\Delta$, where $e$ is a
non-negative rational number, and $\Delta$ is an~effective
$\mathbb{Q}$-divisor, whose support does not contain the conic
$C$. Then
$$
\bar{\Delta}\equiv\alpha^{*}\big(\Delta\big)-\delta E_{1},
$$
where $\bar{\Delta}$ is the proper transforms of $\Delta$ on
the~surface $\bar{S}$, and $\delta=a_{1}-e/2$. Then
$$
2-e-\delta=\bar{Z}\cdot\bar{\Delta}>1/\lambda-e/2-\delta>2-e/2-\delta
$$
by Remark~\ref{remark:adjunction}, because $\bar{C}^{2}=1/2$. We
have $e<0$, which is impossible.

We see that the conic $Z$ is reducible. Then
$$
Z=L_{2}+L_{2}^{\prime},
$$
where $L_{2}^{\prime}$ is a line on $S$ such that $P\in
L_{2}^{\prime}$ and $L_{2}\cap L_{2}^{\prime}\ne\varnothing$.

The intersection $L_{2}\cap L_{2}^{\prime}$ consists of a single
point. The impossibility of the case $Q\in\bar{L}_{1}$~implies
that the surface $S$ is smooth at the point $L_{2}\cap
L_{2}^{\prime}$. There is a rational number $c\geqslant 0$ such
that
$$
D=cL_{2}+\Xi,
$$
where $\Xi$ is an~effective $\mathbb{Q}$-divisor, whose support
does not contain the line $L_{2}$. Then
$$
\bar{\Xi}\equiv\alpha^{*}\big(\Xi\big)-\upsilon E_{1},
$$
where $\bar{\Xi}$ is the proper transforms of  $\Xi$ on
the~surface $\bar{S}$, and $\upsilon=a_{1}-c/2$. The log pair
$$
\Big(\bar{S},\ \lambda c\bar{L}_{2}+\lambda\bar{\Xi}+\lambda(c/2+\upsilon)E_{1}\Big)%
$$
is not log canonical at $Q$. We have $Q\in\bar{L}_{2}$ and
$\bar{L}_{2}^{2}=-1$. Then
$$
1+c/2-\upsilon=\bar{L}_{2}\cdot\bar{\Xi}>1/\lambda-c/2-\upsilon>2-c/2-\upsilon
$$
by Remark~\ref{remark:adjunction}. Therefore, the inequality $c>1$
holds.

There is a unique hyperplane section $T$ of the surface $S$ such
that $T=C_{2}+L_{2}$ and
$$
Q=\bar{C}_{2}\cap \bar{L}_{2}=O,
$$
where $C_{2}$ is a conic, and $\bar{C}_{2}$ is the proper
transforms of $C_{2}$ on the~surface $\bar{S}$.

The conic $C_{2}$ is irreducible. We may assume that
$C_{2}\not\subseteq\mathrm{Supp}(D)$ (see
Remark~\ref{remark:convexity}). Then
$$
2-3c/2-\upsilon=\bar{C}_{2}\cdot\bar{\Xi}\geqslant\mathrm{mult}_{Q}\big(\bar{\Xi}\big)>1/\lambda-c/2-\upsilon,
$$
which implies that $c<0$. The obtained contradiction completes the
proof.
\end{proof}

\begin{lemma}
\label{lemma:dP3-D4} Suppose that $\Sigma=\{\mathbb{D}_{4}\}$.
Then $\mathrm{lct}(S)=1/3$.
\end{lemma}

\begin{proof}
We have $r=3$, and $L_{1}$, $L_{2}$, $L_{3}$ lie in a single
plane. Then
$$
\left(S,\ \frac{1}{3}\Big(L_{1}+L_{2}+L_{3}\Big)\right)
$$
is log canonical and not log terminal. We may assume that
$L_{3}\not\subseteq\mathrm{Supp}(D)$ due to
Remark~\ref{remark:convexity}.

Let $\beta\colon\breve{S}\to S$ be a birational morphism such that
the morphism $\alpha$ contracts one irreducible rational curve $E$
that contains three singular points $O_{1}$, $O_{2}$, $O_{3}$ of
type $\mathbb{A}_{1}$.

Let $\breve{D}$ and $\breve{L}_{i}$ be proper transforms of $D$
and $L_{i}$ on the~surface $\breve{S}$, respectively. Then
$$
\breve{L}_{i}\equiv\beta^{*}(L_{i})-E,\
\breve{D}\equiv\beta^{*}\big(D\big)-\mu E,
$$
where $\mu$ is a rational number. We have
$$
0\leqslant\breve{D}\cdot\breve{L}_{3}=\Big(\beta^{*}\big(D\big)-\mu E\Big)\cdot\breve{L}_{3}=1-\mu E\cdot\breve{L}_{3}=1-\mu/2,%
$$
which implies that $\mu\leqslant 2$. Therefore, we may assume
there is a point $Q\in E$ such that the~singularities of the~log
pair $(\breve{S},\lambda\breve{D}+\breve\mu E)$ are not log
canonical at the~point $Q$ (see
Lemma~\ref{lemma:smooth-points-1-3}).

Suppose that $\breve{S}$ is smooth at $Q$. The log pair
$(\breve{S},\lambda\breve{D}+E)$ is not log canonical at $Q$. Then
$$
1\geqslant\mu/2=-\mu E^{2}=E\cdot\breve{D}>1/\lambda>3%
$$
by Remark~\ref{remark:adjunction}. We see that $Q=O_{j}$ for some
$j$.

The curves $\breve{L}_{1}$, $\breve{L}_{2}$ and $\breve{L}_{3}$
are disjoined, and each of them passes through a singular point of
the~surface $\breve{S}$. Therefore, we may assume that
$O_{i}\in\breve{L}_{i}$ for every $i$.

Let $\gamma\colon\grave{S}\to\breve{S}$ be a blow up of the point
$O_{j}$, and $G$ be the~exceptional curve of $\gamma$. Then
$$
\grave{L}_{j}\equiv\gamma^{*}\big(\breve{L}_{j}\big)-\frac{1}{2}G\equiv\big(\beta\circ\gamma\big)^{*}\big(L_{1}\big)-\grave{E}-G,%
$$
where $\grave{L}_{j}$ and $\grave{E}$ are proper transforms of
the~curves $\bar{L}_{j}$ and $E$ on the~surface $\grave{S}$,
respectively.

Let $\grave{D}$ be the~proper transform of the~divisor $\breve{D}$
on the~surface $\grave{S}$. Then
$$
\grave{E}\equiv\gamma^{*}(E)-\frac{1}{2}G,\ \grave{D}\equiv\gamma^{*}\big(\breve{D}\big)-\epsilon G\equiv\big(\beta\circ\gamma\big)^{*}\big(D\big)-\mu\grave{E}-\big(\mu/2+\epsilon\big)G,%
$$
where $\epsilon$ is a rational number. Then
$\lambda\epsilon+\lambda\mu/2>1/2$ (see Remark~\ref{remark:A1}).

Suppose that $j=3$. Then
$$
0\leqslant\grave{D}\cdot\grave{L}_{3}=\Big(\big(\beta\circ\gamma\big)^{*}\big(D\big)-\mu\grave{E}-\big(\mu/2+\epsilon\big)G\Big)\cdot\grave{L}_{3}=1-\mu/2-\epsilon
$$
which implies that $\epsilon+\mu/2<1$. But we know that
$\epsilon+\mu/2>3/2$.

We may assume that $Q=O_{1}$, and the~support of the~divisor $D$
contains the~line $L_{1}$. Put
$$
D=aL_{1}+\Omega\equiv -K_{S},
$$
where $a\in\mathbb{Q}$ and $a\geqslant 0$, and $\Omega$ is an
effective $\mathbb{Q}$-divisor such that
$L_{1}\not\subseteq\mathrm{Supp}(\Omega)$. Then
$$
\grave{\Omega}\equiv\big(\beta\circ\gamma\big)^{*}\big(\Omega\big)-m\grave{E}-\big(m/2+b\big)G,%
$$
where $\grave{\Omega}$ is the~proper transform of $\Omega$, and
$m$ and $b$ are non-negative rational numbers. Then
$$
\big(\beta\circ\gamma\big)^{*}\big(D\big)-\mu\grave{E}-\big(\mu/2+\epsilon\big)G\equiv\grave{D}=a\grave{L}_{1}+\grave{\Omega}\equiv\big(\beta\circ\gamma\big)^{*}\big(aL_{1}+\Omega\big)-\big(a+m\big)\grave{E}-\big(a+m/2+b\big)G,
$$
which implies that $\mu=a+m\leqslant 2$ and $\epsilon=a/2+b$. We
have
$$
\grave{L}_{1}^{2}=-1,\ \grave{E}^{2}=-1,\ G^{2}=-2,\ \grave{L}\cdot\grave{E}=0,\ \grave{L}\cdot G=\grave{E}\cdot G=1%
$$
on the~surface $\grave{S}$. The surface $\grave{S}$ is smooth
along the~curve $G$. Then
$$
-a\leqslant-a+\grave{\Omega}\cdot\grave{L}_{1}=\Big(a\grave{L}_{1}+\grave{\Omega}\Big)\cdot\grave{L}_{1}=\Big(\big(\beta\circ\gamma\big)^{*}\big(-K_{S}\big)-\big(a+m\big)\grave{E}-\big(a+m/2+b\big)G\Big)\cdot\grave{L}_{1}=1-a-m/2-b,
$$
which implies that $m/2+b\leqslant 1$ and $a+m/2+b\leqslant
1+a\leqslant 3$. Thus, the~equivalence
$$
K_{\grave{S}}+\lambda a\grave{L}_{1}+\lambda\grave{\Omega}+\lambda\big(a+m\big)\grave{E}+\lambda\big(a+m/2+b\big)G\equiv\big(\beta\circ\gamma\big)^{*}\big(K_{S}+\lambda aL_{1}+\lambda\Omega\big)%
$$
implies the~existence of a point $A\in G$ such that the~log pair
$$
\Big(\grave{S},\ \lambda a\grave{L}_{1}+\lambda\grave{\Omega}+\lambda\big(a+m\big)\grave{E}+\lambda\big(a+m/2+b\big)G\Big)%
$$
is not log canonical at the~point $A$.

Suppose that $A\not\in\grave{L}_{1}\cup\grave{E}$. Then
$(\grave{S}, \lambda\grave{\Omega}+\lambda(a+m/2+b)G)$ is not log
canonical at $A$, and
$$
2b+a=\Big(a\grave{L}_{1}+\grave{\Omega}\Big)\cdot G=a+\grave{\Omega}\cdot G>a+3,%
$$
by Remark~\ref{remark:adjunction}. We see that $b>3/2$. But
$m/2+b\leqslant 1$. We see that $A\in\grave{L}_{1}\cup\grave{E}$.

Suppose that $A\not\in\grave{L}_{1}$. The  log pair $(\grave{S},
\lambda\grave{\Omega}+\lambda(a+m)\grave{E}+\lambda(a+m/2+b)G)$ is
not log canonical at the~point $A$. Arguing as in the~previous
case, we see that
$$
m/2-b=\Big(a\grave{L}_{1}+\grave{\Omega}\Big)\cdot\grave{E}=\grave{\Omega}\cdot\grave{E}\geqslant \mathrm{mult}_{A}\Big(\grave{\Omega}\big\vert_{\grave{E}}\Big)>3-\big(a+m/2+b\big),%
$$
which implies that $a+m>3$. But $a+m\leqslant 2$. We see that
$A\in\grave{L}_{1}$.

The log pair $(\grave{S}, \lambda
a\grave{L}_{1}+\lambda\grave{\Omega}+\lambda(a+m/2+b)G)$ is not
log canonical at the~point $A$. Then
$$
1-a-m/2-b=\Big(a\grave{L}_{1}+\grave{\Omega}\Big)\cdot\grave{L}_{1}=-a+\grave{\Omega}\cdot\grave{L}_{1}>-a+3-\big(a+m/2+b\big)%
$$
by Remark~\ref{remark:adjunction}. We have $a>2$. But
$a+m\leqslant 2$.
\end{proof}

\begin{lemma}
\label{lemma:dP3-D5} Suppose that $\Sigma=\{\mathbb{D}_{5}\}$.
Then $\mathrm{lct}(S)=1/4$.
\end{lemma}

\begin{proof}
We have $r=2$. We may assume that $-K_{S}\sim 2L_{1}+L_{2}$. Then
the log pair
$$
\left(S,\ \frac{1}{4}\Big(2L_{1}+L_{2}\Big)\right),
$$
is not log terminal. We may assume $(S,\lambda D)$ is not log
canonical~at~$O$ (see Lemma~\ref{lemma:smooth-points-1-3}).

It follows from \cite{BW79} that $S$ contains a line $L$ such that
$O\not\in L$. Projecting from $L$, we~see~that there is a conic
$C\subset S$  such that $O\not\in C$, $-K_{S}\sim C+L$, and
$C\cdot L=2$. Put $P=C\cap L$. Then
$$
P\cup O\subseteq\mathrm{LCS}\left(S,\ \frac{3}{4}\Big(C+L\Big)+\lambda D\right)\subseteq P\cup O\cup C\cup L,%
$$
which is impossible by Remark~\ref{remark:connectedness}.
\end{proof}

\begin{lemma}
\label{lemma:dP3-E6} Suppose that $\Sigma=\{\mathbb{E}_{6}\}$.
Then $\mathrm{lct}(S)=1/6$.
\end{lemma}

\begin{proof}
We have $r=1$. The log pair $(S, \frac{1}{2}L_{1})$ is not
log~ter\-mi\-nal. The surface $S$ contains a~plane cuspidal curve
$C$ such that $O\not\in C$. The proof of Lemma~\ref{lemma:dP3-D4}
implies that $\mathrm{lct}(S)=1/6$.
\end{proof}

\begin{lemma}
\label{lemma:dP3-A2} Suppose that $\Sigma=\{\mathbb{A}_{2}\}$.
Then $\mathrm{lct}(S)=1/2$.
\end{lemma}

\begin{proof}
We may assume that $-K_{S}\sim L_{1}+L_{2}+L_{3}\sim
L_{4}+L_{5}+L_{6}$. The log pair
$$
\left(S,\ \frac{1}{2}\Big(L_{1}+L_{2}+L_{3}\Big)\right)
$$
is log canonical and not log terminal. Hence, we may assume that
$(S,\lambda D)$ is not log canonical.

The proof of Lemma~\ref{lemma:dP3-A1-single} implies that
$\mathrm{LCS}(S,\lambda D)=O$.

Let $\bar{H}$ be a proper transform on $\bar{S}$ of a general
hyperplane section that contains $O$. Then
$$
0\leqslant\bar{H}\cdot\bar{D}=3-a_{1}-a_{2},\ 2a_{1}-a_{2}=E_{1}\cdot\bar{D}\geqslant 0,\ 2a_{2}-a_{1}=E_{2}\cdot\bar{D}\geqslant 0,%
$$
which implies that $a_{1}\leqslant 2$ and $a_{2}\leqslant 2$.
There is a point $Q\in E_{1}\cup E_{2}$ such that the
singularities of the log pair
$(\bar{S},\lambda(\bar{D}+a_{1}E_{1}+a_{2}E_{2}))$ are not log
canonical at $Q$. We may assume that~$Q\in E_{1}$,~and
$$
\bar{L}_{1}\cdot E_{1}=\bar{L}_{2}\cdot E_{1}=\bar{L}_{3}\cdot E_{1}=\bar{L}_{4}\cdot E_{2}=\bar{L}_{5}\cdot E_{2}=\bar{L}_{6}\cdot E_{2}=1,%
$$
which implies that $\bar{L}_{1}\cdot E_{2}=\bar{L}_{2}\cdot
E_{2}=\bar{L}_{3}\cdot E_{2}=\bar{L}_{4}\cdot
E_{1}=\bar{L}_{5}\cdot E_{1}=\bar{L}_{6}\cdot E_{1}=0$.

It follows from Remark~\ref{remark:convexity} that we may assume
that
$\bar{L}_{1}\not\subseteq\mathrm{Supp}(D)\not\supseteq\bar{L}_{4}$.
Then
$$
1-a_{1}=\bar{D}\cdot\bar{L}_{1}\geqslant 0,\ 1-a_{2}=\bar{D}\cdot\bar{L}_{4}\geqslant 0,%
$$
which implies that $a_{1}\leqslant 1$ and $a_{2}\leqslant 1$.

Suppose that $Q\not\in E_{2}$. Then $(\bar{S},
\lambda\bar{D}+E_{1})$ is not log canonical at $Q$. We have
$$
2a_{1}-a_{2}=\bar{D}\cdot E_{1}>1/\lambda>2,
$$
by Remark~\ref{remark:adjunction}. Then $a_{1}\geqslant 4/3$,
which is impossible. Hence, we see that $Q\in E_{2}$.

The log pairs $(\bar{S}, \lambda\bar{D}+E_{1}+a_{2}E_{2})$ and
$(\bar{S}, \lambda\bar{D}+a_{1}E_{1}+E_{2})$ are not log canonical
at $Q$. Then
$$
2a_{1}-a_{2}=\bar{D}\cdot E_{1}>1/\lambda-a_{2}>2-a_{2},\ 2a_{2}-a_{1}=\bar{D}\cdot E_{2}>1/\lambda-a_{1}>2-a_{1}%
$$
by
Remark~\ref{remark:adjunction}. Then $a_{1}>1$ and $a_{2}>1$,
which is impossible.
\end{proof}

\begin{lemma}
\label{lemma:dP3-A3} Suppose that $\Sigma=\{\mathbb{A}_{3}\}$.
Then $\mathrm{lct}(S)=1/2$.
\end{lemma}

\begin{proof}
We have $r=5$. Then $\bar{L}_{i}^{2}=-1$ and
$\bar{L}_{i}\cdot\bar{L}_{j}=0$ for $i\ne j$. We may assume~that
$$
\bar{L}_{1}\cdot E_{1}=\bar{L}_{2}\cdot E_{1}=\bar{L}_{3}\cdot E_{2}=\bar{L}_{4}\cdot E_{3}=\bar{L}_{5}\cdot E_{3}=1,%
$$
which implies that $\bar{L}_{3}\cdot E_{1}=\bar{L}_{3}\cdot
E_{2}=0$ and
$$
\bar{L}_{1}\cdot E_{2}=\bar{L}_{2}\cdot E_{2}=\bar{L}_{1}\cdot E_{3}=\bar{L}_{2}\cdot E_{3}=\bar{L}_{4}\cdot E_{2}=\bar{L}_{5}\cdot E_{2}=\bar{L}_{4}\cdot E_{1}=\bar{L}_{5}\cdot E_{1}=0.%
$$

We have $-K_{S}\sim L_{1}+L_{2}+L_{3}$. But it follows from
elementary calculations that
$$
\bar{L}_{1}\equiv\alpha^{*}\big(L_{1}\big)-\frac{3}{4}E_{1}-\frac{1}{2}E_{2}-\frac{1}{4}E_{3},\ \bar{L}_{2}\equiv\alpha^{*}\big(L_{2}\big)-\frac{3}{4}E_{1}-\frac{1}{2}E_{2}-\frac{1}{4}E_{3},\ \bar{L}_{3}\equiv\alpha^{*}\big(L_{3}\big)-\frac{1}{2}E_{1}-E_{2}-\frac{1}{2}E_{3},%
$$
which implies that $\mathrm{lct}(S)\leqslant 1/2$. Hence, we may
assume that $(S,\lambda D)$ is not log canonical.

Suppose that there are a line $L\subset S$ and a rational number
$\mu\geqslant 1/\lambda$ such that $D=\mu L+\Omega$, where
$\Omega$ is an~effective $\mathbb{Q}$-divisor, whose support does
not contain the line $L$.  Then
$$
2=C\cdot D=\mu C\cdot L+C\cdot\Omega\geqslant \mu C\cdot L>2 C\cdot L,%
$$
where $C$ is a general conic on the surface $S$ such that the
divisor $C+L$ is a hyperplane section of the surface $S$. Then
$|L\cap C|=1$, which implies that $L=L_{3}$. But $L_{3}\cdot C=1$.

It follows from Remark~\ref{remark:connectedness} and
Lemma~\ref{lemma:smooth-points-LCS} that $\mathrm{LCS}(S,\lambda
D)=O$.

Let $\bar{H}$ be a general curve in
$|-K_{\bar{S}}-\sum_{i=1}^{3}E_{i}|$. Then
$$
a_{1}+a_{3}\leqslant 3,\ 2a_{1}\geqslant a_{2},\ 2a_{2}\geqslant a_{1}+a_{3},\ 2a_{3}\geqslant a_{2},%
$$
because $\bar{H}\cdot\bar{D}\geqslant 0$,
$E_{1}\cdot\bar{D}\geqslant 0$, $E_{2}\cdot\bar{D}\geqslant 0$,
$E_{3}\cdot\bar{D}\geqslant 0$, respectively.

We may assume that either $L_{1}\not\subseteq\mathrm{Supp}(D)$ or
$L_{3}\not\subseteq\mathrm{Supp}(D)$. But
$$
\bar{L}_{1}\cdot\bar{D}=1-a_{1},\ \bar{L}_{3}\cdot\bar{D}=1-a_{2},%
$$
which implies that either $a_{1}\leqslant 1$ or $a_{2}\leqslant
1$. Similarly, we assume that either $a_{3}\leqslant 1$ or
$a_{2}\leqslant 1$.

We have $a_{1}\leqslant 2$, $a_{2}\leqslant 2$, $a_{3}\leqslant
2$. Then there is a point $Q\in E_{1}\cup E_{2}\cup E_{3}$ such
that the log pair
$(\bar{S},\lambda(\bar{D}+a_{1}E_{1}+a_{2}E_{2}+a_{3}E_{3}))$ is
not log canonical at $Q$. We may assume that $Q\not\in E_{3}$.

Suppose that $Q\not\in E_{2}$. Then $(\bar{S},
\lambda\bar{D}+E_{1})$ is not log canonical at the point $Q$. We
have
$$
2a_{1}-a_{2}=\bar{D}\cdot E_{1}>1/\lambda>2,
$$
by Remark~\ref{remark:adjunction}. Then $a_{1}>3/2$ and $a_{2}>1$.
But either $a_{1}\leqslant 1$ or $a_{2}\leqslant 1$.
Contradiction.

Suppose that $Q\in E_{2}\cap E_{1}$. Arguing as in the proof of of
Lemma~\ref{lemma:dP3-A2}, we see that
$$
2a_{1}-a_{2}=\bar{D}\cdot E_{1}>1/\lambda-a_{2}>2-a_{2},\ 2a_{2}-a_{1}-a_{3}=\bar{D}\cdot E_{2}>1/\lambda-a_{1}>2-a_{1}%
$$
by Remark~\ref{remark:adjunction}. Then $a_{1}>1$ and
$2a_{2}>2+a_{3}$, which is impossible.

We see that $Q\in E_{2}$ and $Q\not\in E_{1}$. Then $(\bar{S},
\lambda\bar{D}+E_{2})$ is not log canonical at $Q$. We have
$$
2a_{2}-a_{1}-a_{3}=\bar{D}\cdot E_{2}>1/\lambda>2,
$$
which implies that $a_{1}>3/2$ and $a_{2}>2$. The latter is
impossible.
\end{proof}

\begin{lemma}
\label{lemma:dP3-A4} Suppose that $\Sigma=\{\mathbb{A}_{4}\}$.
Then $\mathrm{lct}(S)=1/3$.
\end{lemma}

\begin{proof}
We have $r=4$. Then $\bar{L}_{i}^{2}=-1$ and
$\bar{L}_{i}\cdot\bar{L}_{j}=0$ for $i\ne j$. We may assume~that
$$
\bar{L}_{1}\cdot E_{1}=\bar{L}_{2}\cdot E_{1}=\bar{L}_{3}\cdot E_{3}=\bar{L}_{4}\cdot E_{4}=1,%
$$
which implies that $\bar{L}_{3}\cdot E_{1}=\bar{L}_{3}\cdot
E_{2}=\bar{L}_{3}\cdot E_{4}=0$ and
$$
\bar{L}_{1}\cdot E_{2}=\bar{L}_{2}\cdot E_{2}=\bar{L}_{1}\cdot E_{3}=\bar{L}_{2}\cdot E_{3}=\bar{L}_{1}\cdot E_{4}=\bar{L}_{2}\cdot E_{4}=\bar{L}_{4}\cdot E_{1}=\bar{L}_{4}\cdot E_{2}=\bar{L}_{4}\cdot E_{3}=0.%
$$

The equivalence $-K_{S}\sim 2L_{3}+L_{4}$ holds. Similarly, we
have
$$
\bar{L}_{3}\equiv\alpha^{*}\big(L_{3}\big)-\frac{2}{5}E_{1}-\frac{4}{5}E_{2}-\frac{6}{5}E_{3}-\frac{3}{5}E_{4},\ \bar{L}_{4}\equiv\alpha^{*}\big(L_{4}\big)-\frac{1}{5}E_{1}-\frac{2}{5}E_{2}-\frac{3}{5}E_{3}-\frac{4}{5}E_{4},%
$$
which implies that $\mathrm{lct}(S)\leqslant 1/3$. Thus, we may
assume that $(S,\lambda D)$ is not log canonical, which implies
that $\mathrm{LCS}(S,\lambda D)=O$ by
Lemma~\ref{lemma:smooth-points-1-3}. Let $\bar{H}$ be a general
curve in $|-K_{\bar{S}}-\sum_{i=1}^{4}E_{i}|$.~Then
$$
3\geqslant a_{1}+a_{4},\ 2a_{1}\geqslant a_{2},\ 2a_{2}\geqslant a_{1}+a_{3},\ 2a_{3}\geqslant a_{2}+a_{4},\ 2a_{4}\geqslant a_{3},%
$$
because $\bar{H}\cdot\bar{D}\geqslant 0$,
$E_{1}\cdot\bar{D}\geqslant 0$, $E_{2}\cdot\bar{D}\geqslant 0$,
$E_{3}\cdot\bar{D}\geqslant 0$, $E_{4}\cdot\bar{D}\geqslant 0$,
respectively.

We have $-K_{S}\sim L_{1}+L_{2}+L_{3}$ and $-K_{S}\sim
2L_{3}+L_{4}$. But the log pairs
$$
\left(S,\ \frac{1}{2}\Big(L_{1}+L_{2}+L_{3}\Big)\right)\ \text{and}\ \left(S,\ \frac{1}{3}\Big(L_{4}+2L_{3}\Big)\right)%
$$
are log canonical. So, we may assume that either
$L_{3}\not\subseteq\mathrm{Supp}(D)$ or
$L_{1}\not\subseteq\mathrm{Supp}(D)\not\supseteq L_{4}$. But
$$
\bar{L}_{3}\cdot\bar{D}=1-a_{3},\ \bar{L}_{1}\cdot\bar{D}=1-a_{1},\ \bar{L}_{4}\cdot\bar{D}=1-a_{4},%
$$
which implies that there is a point $Q\in\cup_{i=1}^{4}E_{i}$ such
that $(\bar{S},\lambda(\bar{D}+\sum_{i=1}^{4}a_{i}E_{i}))$ is not
log canonical at the point $Q$. Arguing as in the proof of
Lemma~\ref{lemma:dP3-A3}, we see that
$$
\left\{%
\aligned
&Q\in E_{1}\setminus (E_{1}\cap E_{2})\Rightarrow 2a_{1}>a_{2}+3,\\%
&Q\in E_{1}\cap E_{2}\Rightarrow 2a_{1}>3\ \text{and}\ 2a_{2}>3+a_{3},\\%
&Q\in E_{2}\setminus \big((E_{1}\cap E_{2})\cup(E_{2}\cap E_{3})\big)\Rightarrow 2a_{2}>a_{1}+a_{3}+3,\\%
&Q\in E_{2}\cap E_{3}\Rightarrow 2a_{2}>3+a_{1}\ \text{and}\ 2a_{3}>3+a_{4},\\%
&Q\in E_{3}\setminus \big((E_{2}\cap E_{3})\cup(E_{3}\cap E_{4})\big)\Rightarrow 2a_{3}>3+a_{2}+a_{4},\\%
&Q\in E_{3}\cap E_{4}\Rightarrow 2a_{3}>3+a_{2}\ \text{and}\ 2a_{4}>3,\\%
&Q\in E_{4}\setminus (E_{4}\cap E_{3})\Rightarrow 2a_{4}>3,\\%
\endaligned\right.%
$$
which leads to a contradiction, because  either $a_{3}\leqslant 1$
or $a_{1}\leqslant 1$ and $a_{4}\leqslant 1$.
\end{proof}

\begin{lemma}
\label{lemma:dP3-A5} Suppose that $\Sigma=\mathbb{A}_{5}$. Then
$\mathrm{lct}(S)=1/4$.
\end{lemma}

\begin{proof}
We have $r=3$. We may assume~that $\bar{L}_{1}\cdot
E_{1}=\bar{L}_{2}\cdot E_{1}=\bar{L}_{3}\cdot E_{4}=1$. Then
$$
\bar{L}_{1}\cdot E_{2}=\bar{L}_{2}\cdot E_{2}=\bar{L}_{1}\cdot E_{3}=\bar{L}_{2}\cdot E_{3}=\bar{L}_{1}\cdot E_{4}=\bar{L}_{2}\cdot E_{4}=\bar{L}_{1}\cdot E_{5}=\bar{L}_{2}\cdot E_{3}=0%
$$
and $\bar{L}_{3}\cdot E_{1}=\bar{L}_{3}\cdot
E_{2}=\bar{L}_{3}\cdot E_{3}=\bar{L}_{3}\cdot E_{5}=0$. But
$-K_{S}\sim 3L_{3}$. Then $\mathrm{lct}(S)\leqslant 1/4$, because
$$
\bar{L}_{3}\equiv\alpha^{*}\big(L_{3}\big)-\frac{1}{3}E_{1}-\frac{2}{3}E_{2}-E_{3}-\frac{4}{3}E_{4}-\frac{2}{3}E_{5}.
$$

We may assume that $(S,\lambda D)$ is not log canonical. Then
$\mathrm{LCS}(S,\lambda D)=O$. by
Lemma~\ref{lemma:smooth-points-1-3}.

Let $\bar{H}$ be a proper transform on $\bar{S}$ of a general
hyperplane section that contains $O$. Then
\begin{equation}
\label{equation:A5-simple}
3\geqslant a_{1}+a_{5},\ 2a_{1}\geqslant a_{2},\ 2a_{2}\geqslant a_{1}+a_{3},\ 2a_{3}\geqslant a_{2}+a_{4},\ 2a_{4}\geqslant a_{3}+a_{5},\ 2a_{5}\geqslant a_{4},%
\end{equation}
because $\bar{H}\cdot\bar{D}\geqslant 0$,
$E_{1}\cdot\bar{D}\geqslant 0$, $E_{2}\cdot\bar{D}\geqslant 0$,
$E_{3}\cdot\bar{D}\geqslant 0$, $E_{4}\cdot\bar{D}\geqslant 0$,
$E_{5}\cdot\bar{D}\geqslant 0$, respectively.

We may assume that $L_{3}\not\subseteq\mathrm{Supp}(D)$ due to
Remark~\ref{remark:convexity}. Then
$1-a_{4}=\bar{L}_{3}\cdot\bar{D}\geqslant 0$, which~easily implies
that $a_{1}\leqslant 5/2$, $a_{2}\leqslant 2$, $a_{3}\leqslant
3/2$, $a_{4}\leqslant 1$, $a_{5}\leqslant 5/4$.

There is a point $Q\in \cup_{i=1}^{5}E_{i}$ such that the log pair
$(\bar{S},\lambda(\bar{D}+\sum_{i=1}^{5}a_{i}E_{i}))$ is~not log
canonical at the point $Q$. Arguing as in the proof of
Lemma~\ref{lemma:dP3-A3}, we see that
\begin{equation}
\label{equation:A5}
\left\{%
\aligned
&Q\in E_{1}\setminus (E_{1}\cap E_{2})\Rightarrow 2a_{1}>a_{2}+4,\\%
&Q\in E_{1}\cap E_{2}\Rightarrow 2a_{1}>4\ \text{and}\ 2a_{2}>4+a_{3},\\%
&Q\in E_{2}\setminus \big((E_{1}\cap E_{2})\cup(E_{2}\cap E_{3})\big)\Rightarrow 2a_{2}>a_{1}+a_{3}+4,\\%
&Q\in E_{2}\cap E_{3}\Rightarrow 2a_{2}>4+a_{1}\ \text{and}\ 2a_{3}>4+a_{4},\\%
&Q\in E_{3}\setminus \big((E_{2}\cap E_{3})\cup(E_{3}\cap E_{4})\big)\Rightarrow 2a_{3}>4+a_{2}+a_{4},\\%
&Q\in E_{3}\cap E_{4}\Rightarrow 2a_{3}>4+a_{2}\ \text{and}\ 2a_{4}>4+a_{5},\\%
&Q\in E_{4}\setminus \big((E_{3}\cap E_{4})\cup(E_{4}\cap E_{5})\big)\Rightarrow 2a_{4}>4+a_{3}+a_{5},\\%
&Q\in E_{4}\cap E_{5}\Rightarrow 2a_{4}>4+a_{3}\ \text{and}\ 2a_{5}>4,\\%
&Q\in E_{5}\setminus (E_{4}\cap E_{5})\Rightarrow 2a_{5}>a_{4}+4.\\%
\endaligned\right.%
\end{equation}

Now taking into account the inequalities~\ref{equation:A5-simple},
the inequalities~\ref{equation:A5}, the~inequality $a_{4}\leqslant
4$, and the~inequality $a_{1}+a_{5}\leqslant 3$, we see that
either $Q=E_{3}\cap E_{4}$ or $Q=E_{4}\cap E_{5}$.

Let $H_{1}$ and $H_{3}$ be general curves in $|-K_{S}|$ that
contain $L_{1}$ and $L_{3}$, respectively. Then
$$
H_{1}=L_{1}+C_{1},\ H_{3}=L_{3}+C_{3},
$$
where $C_{1}$ and $C_{3}$ are irreducible conics such that
$C_{1}\not\subseteq\mathrm{Supp}(D)\not\supseteq C_{3}$.

Let $\bar{C}_{1}$ and $\bar{C}_{3}$ be the proper transforms of
$C_{1}$ and $C_{3}$ on the surface $\bar{S}$, respectively. Then
$$
\bar{C}_{1}\cdot E_{1}=\bar{C}_{1}\cdot E_{2}=\bar{C}_{1}\cdot E_{3}=\bar{C}_{1}\cdot E_{4}=\bar{C}_{3}\cdot E_{1}=\bar{C}_{3}\cdot E_{3}=\bar{C}_{3}\cdot E_{4}=\bar{C}_{3}\cdot E_{5}=0%
$$
and $\bar{C}_{1}\cdot E_{5}=\bar{C}_{3}\cdot E_{2}=1$. Therefore,
we see that
$$
0\leqslant \bar{C}_{1}\cdot\bar{D}=2-a_{5},\ 2-a_{2}=\bar{C}_{3}\cdot\bar{D}\geqslant 0,%
$$
which implies that $a_{2}\leqslant 2$ and $a_{5}\leqslant 2$. Now
we can easily obtain a contradiction.
\end{proof}

\begin{lemma}
\label{lemma:dP3-A5-A1} Suppose that $\Sigma=\{\mathrm{A}_{1},
\mathbb{A}_{5}\}$. Then $\mathrm{lct}(S)=1/4$.
\end{lemma}

\begin{proof}
Let $P$ be a point in $\Sigma$ of type $\mathbb{A}_{1}$. Then
$r=2$. We may assume that $P\in L_{1}$. Then
$$
\bar{L}_{2}\cdot E_{1}=\bar{L}_{2}\cdot E_{2}=\bar{L}_{2}\cdot E_{3}=\bar{L}_{2}\cdot E_{5}=\bar{L}_{1}\cdot E_{2}=\bar{L}_{1}\cdot E_{3}=\bar{L}_{1}\cdot E_{4}=\bar{L}_{1}\cdot E_{5}=0,%
$$
and $\bar{L}_{1}\cdot E_{1}=\bar{L}_{2}\cdot E_{4}=1$. The
equivalence $-K_{S}\sim 3L_{2}$ holds. Then
$\mathrm{lct}(S)\leqslant 1/4$, because
$$
\bar{L}_{2}\equiv\alpha^{*}\big(L_{2}\big)-\frac{1}{3}E_{1}-\frac{2}{3}E_{2}-E_{3}-\frac{4}{3}E_{4}-\frac{2}{3}E_{5}.
$$

We may assume that $(S,\lambda D)$ is not log canonical. Then
$\mathrm{LCS}(S,\lambda D)\subseteq \{O, P\}$ by
Lemma~\ref{lemma:smooth-points-1-3}.

Suppose that $(S,\lambda D)$ is not log terminal at $P$. Let
$\beta\colon\breve{S}\to S$ be a blow up of $P$. Then
$$
\breve{D}\equiv\beta^{*}\big(-K_{S}\big)-mF,
$$
where $F$ is the $\beta$-exceptional curve, $\breve{D}$ is the
proper transform of the divisor $D$, and $m\in\mathbb{Q}$.~Then
$$
0\leqslant\breve{H}\cdot\breve{D}=\Big(\beta^{*}\big(-K_{S}\big)-mF\Big)\cdot\Big(\beta^{*}\big(-K_{S}\big)-F\Big)=3-2m,
$$
where $\breve{H}$ is general curve in $|-K_{\breve{S}}-F|$. Thus,
we have $m\leqslant 3/2$. But $m>2$ by Remark~\ref{remark:A1}.

We see that $\mathrm{LCS}(S,\lambda D)=O$. Let $C_{1}$ and $C_{2}$
be general conics on the surface $S$ such that
$$
L_{1}+C_{1}\sim L_{2}+C_{2}\sim -K_{S},
$$
and let $\bar{C}_{1}$ and $\bar{C}_{2}$ be the proper transforms
of $C_{1}$ and $C_{2}$ on the surface $\bar{S}$, respectively.
Then
$$
2-a_{1}=\bar{C}_{1}\cdot \bar{D}\geqslant 0,\ 2-a_{5}=\bar{C}_{2}\cdot \bar{D}\geqslant 0,%
$$
because $C_{1}\not\subseteq\mathrm{Supp}(D)\not\supseteq C_{2}$.
We may assume that $L_{2}\not\subseteq\mathrm{Supp}(D)$ due to
Remark~\ref{remark:convexity}.

Arguing as in the proof of Lemma~\ref{lemma:dP3-A5}, we obtain the
inequalities
$$
3\geqslant a_{1}+a_{5},\ 2a_{1}\geqslant a_{2},\ 2a_{2}\geqslant a_{1}+a_{3},\ 2a_{3}\geqslant a_{2}+a_{4},\ 2a_{4}\geqslant a_{3}+a_{5},\ 2a_{5}\geqslant a_{4},\ 2\geqslant a_{2},\ 2\geqslant a_{5},\ 1\geqslant a_{4},%
$$
which imply that there is a point $Q\in \cup_{i=1}^{5}E_{i}$ such
that $(\bar{S},\lambda(\bar{D}+\sum_{i=1}^{5}a_{i}E_{i}))$ is not
log canonical at the point $Q$. Arguing as in the proof of
Lemma~\ref{lemma:dP3-A3}, we obtain a contradiction.
\end{proof}

\begin{lemma}
\label{lemma:dP3-A4-A1} Suppose that $\Sigma=\{\mathrm{A}_{1},
\mathbb{A}_{4}\}$. Then $\mathrm{lct}(S)=1/3$.
\end{lemma}

\begin{proof}
We have $r=3$. Let $P$ be a point in $\Sigma$ of type
$\mathbb{A}_{1}$. We may assume that $P\in L_{1}$. Then
$$
\bar{L}_{1}\cdot E_{1}=1,\ \bar{L}_{1}\cdot E_{2}=\bar{L}_{1}\cdot E_{3}=\bar{L}_{1}\cdot E_{4}=0,%
$$
and we may assume~that $\bar{L}_{3}\cdot E_{3}=\bar{L}_{2}\cdot
E_{4}=1$. Then
$$
\bar{L}_{3}\cdot E_{1}=\bar{L}_{3}\cdot E_{2}=\bar{L}_{3}\cdot E_{4}=\bar{L}_{2}\cdot E_{1}=\bar{L}_{2}\cdot E_{2}=\bar{L}_{2}\cdot E_{3}=0.%
$$

The equivalence $-K_{S}\sim L_{2}+2L_{3}$ holds. But
$$
\bar{L}_{2}\equiv\alpha^{*}\big(L_{2}\big)-\frac{1}{5}E_{1}-\frac{2}{5}E_{2}-\frac{3}{5}E_{3}-\frac{4}{5}E_{4},\ \bar{L}_{3}\equiv\alpha^{*}\big(L_{3}\big)-\frac{2}{5}E_{1}-\frac{4}{5}E_{2}-\frac{6}{5}E_{3}-\frac{3}{5}E_{4},%
$$
which implies that $\mathrm{lct}(S)\leqslant 1/3$. Thus, we may
assume that $(S,\lambda D)$ is not log canonical.

We may assume that either $L_{3}\not\subseteq\mathrm{Supp}(D)$ or
$L_{1}\not\subseteq\mathrm{Supp}(D)\not\supseteq L_{2}$ (see
Remark~\ref{remark:convexity}).

Arguing as in the proof of Lemma~\ref{lemma:dP3-A5-A1}, we see
that the log pair $(S,\lambda D)$ is log canonical outside of the
point $O$. Now arguing as in the proof of
Lemma~\ref{lemma:dP3-A4}, we obtain a contradiction.
\end{proof}

\begin{lemma}
\label{lemma:dP3-A3-A1}  Suppose that $\Sigma=\{\mathrm{A}_{1},
\mathbb{A}_{3}\}$. Then $\mathrm{lct}(S)=1/2$.
\end{lemma}

\begin{proof}
Let $P$ be a point in $\Sigma$ of type $\mathbb{A}_{1}$. We may
assume that $P\in L_{1}$. Then $r=4$, and~it~easily follows from
\cite{BW79} that the surface $S$ contains lines
$L_{5},L_{6},L_{7}$ such that
$$
L_{5}\ni P\in L_{6},\ O\not\in L_{7}\not\ni P,\ L_{3}\cap L_{5}\ne\varnothing,\ L_{4}\cap L_{6}\ne\varnothing,\ L_{7}\cap L_{2}\ne\varnothing,\ L_{7}\cap L_{5}\ne\varnothing,\ L_{7}\cap L_{6}\ne\varnothing,%
$$
which implies that $L_{7}\cap L_{1}=L_{7}\cap L_{3}=L_{7}\cap
L_{4}=\varnothing$.~Then $-K_{S}\sim L_{1}+L_{3}+L_{5}$ and
$$
L_{1}+L_{3}+L_{5}\sim L_{1}+L_{4}+L_{6}\sim L_{5}+L_{6}+L_{7}\sim L_{2}+2L_{1}\sim L_{2}+L_{3}+L_{4}\sim 2L_{2}+L_{7},%
$$
which implies that $\mathrm{lct}(S)\leqslant 1/2$. Hence, we may
assume that $(S,\lambda D)$ is not log canonical.

Put $D=\mu_{i} L_{i}+\Omega_{i}$, where $\mu_{i}$ is a
non-negative rational number, and $\Omega_{i}$ is an~effective
$\mathbb{Q}$-di\-vi\-sor, whose support does not contain the line
$L_{i}$. Let us show that that $\mu_{i}<1/\lambda$ for
$i=1,\ldots,7$.

Suppose that $\mu_{2}\geqslant 1/\lambda$. We may assume that
$L_{1}\not\subseteq\mathrm{Supp}(D)$ by
Remark~\ref{remark:convexity}. Then
$$
1=L_{1}\cdot D=L_{1}\cdot\big(\mu_{2}L_{2}+\Omega_{2}\big)\geqslant \mu_{2} L_{1}\cdot L_{2}=\mu_{2}/2>1,%
$$
which is a contradiction.  Similarly, we see that
$\mu_{i}<1/\lambda$ for $i=1,\ldots,7$.

Arguing as in the proof of Lemma~\ref{lemma:dP3-A1-single}, we see
that $\mathrm{LCS}(S,\lambda D)$ does not contain curves and
smooth points of the surface $S$. Then either
$\mathrm{LCS}(S,\lambda D)=O$ or $\mathrm{LCS}(S,\lambda D)=P$ by
Remark~\ref{remark:connectedness}.

Suppose that $\mathrm{LCS}(S,\lambda D)=P$.  Put
$$
D=\mu_{5}L_{5}+\mu_{6}L_{6}+\Upsilon,
$$
where $\Upsilon$ is an~effective $\mathbb{Q}$-divisor such that
$L_{5}\not\subseteq\mathrm{Supp}(\Upsilon)\not\supseteq L_{6}$.
Then $\mu_{5}>0$ and $\mu_{6}>0$.~But
$$
1=L_{7}\cdot
D=L_{7}\cdot\big(\mu_{5}L_{5}+\mu_{6}L_{6}+\Upsilon\big)\geqslant
L_{7}\cdot\big(\mu_{5}L_{5}+\mu_{6}L_{6}\big)=\mu_{5}+\mu_{6},
$$
because we may assume that
$L_{7}\not\subseteq\mathrm{Supp}(\Upsilon)$. Let
$\beta\colon\breve{S}\to S$ be a blow~up~of~the~point~$P$.~Then
$$
\mu_{5}\breve{L}_{5}+\mu_{6}\breve{L}_{6}+\breve{\Upsilon}\equiv\beta^{*}\big(\mu_{5}L_{5}+\mu_{6}L_{6}+\Upsilon\big)-\big(\mu_{5}/2+\mu_{6}/2+\epsilon\big)G,
$$
where $\epsilon$ is a rational number, $G$ is the exceptional
curve of $\beta$, and $\breve{L}_{5}$, $\breve{L}_{6}$,
$\breve{\Upsilon}$~are~proper transforms of the divisors $L_{5}$,
$L_{6}$, $\Upsilon$ on the surface $\breve{S}$, respectively. Then
$$
0\leqslant\Big(\mu_{5}\breve{L}_{5}+\mu_{6}\breve{L}_{6}+\breve{\Upsilon}\Big)\breve{H}=3-\mu_{5}-\mu_{6}-2\epsilon,
$$
where $\breve{H}$ is a general curve in $|-K_{\breve{S}}-G|$.
There is a point $Q\in G$ such that the singularities of the log
pair
$(\breve{S},\lambda(\mu_{5}\breve{L}_{5}+\mu_{6}\breve{L}_{6}+\breve{\Upsilon})+\lambda(\mu_{5}/2+\mu_{6}/2+\epsilon)G)$
are not log canonical at $Q$. We have
$$
2-2\epsilon=\breve{\Upsilon}\cdot\big(\breve{L}_{5}+\breve{L}_{6}\big)\geqslant 0,%
$$
which implies that $\epsilon\leqslant 1$. Then
$2\epsilon=\breve{\Omega}\cdot G>2$ in the case when
$Q\not\in\breve{L}_{5}\cup\breve{L}_{6}$ by
Remark~\ref{remark:adjunction}, which implies that we may assume
that $Q\in\breve{L}_{5}$. Then
$$
1+\mu_{5}/2-\mu_{6}-\epsilon=\breve{\Omega}\cdot\breve{L}_{5}>2-\mu_{5}/2-\mu_{6}/2-\epsilon,%
$$
due to Remark~\ref{remark:adjunction}. Thus, we see that
$\mu_{5}>1$. But $\mu_{5}\leqslant\mu_{5}+\mu_{6}\leqslant 1$.

We see that $\mathrm{LCS}(S,\lambda D)=O$. We may assume that
$$
\bar{L}_{1}\cdot E_{1}=\bar{L}_{2}\cdot E_{2}=\bar{L}_{3}\cdot E_{3}=\bar{L}_{4}\cdot E_{3}=1,
$$
and $\bar{L}_{1}\cdot E_{2}=\bar{L}_{1}\cdot
E_{3}=\bar{L}_{2}\cdot E_{1}=\bar{L}_{2}\cdot
E_{3}=\bar{L}_{3}\cdot E_{1}=\bar{L}_{3}\cdot
E_{2}=\bar{L}_{4}\cdot E_{1}=\bar{L}_{4}\cdot E_{2}=0$.
The~log~pairs
$$
\left(S,\ \frac{1}{2}\Big(2L_{1}+L_{2}\Big)\right)\ \text{and}\ \left(S,\ \frac{1}{2}\Big(L_{2}+L_{3}+L_{3}\Big)\right)%
$$
are log canonical. So, we may assume that either
$L_{2}\not\subseteq\mathrm{Supp}(D)$ or
$L_{1}\not\subseteq\mathrm{Supp}(D)\not\supseteq L_{3}$, which
easily leads to a contradiction (see the proof of
Lemma~\ref{lemma:dP3-A3}).
\end{proof}

\begin{lemma}
\label{lemma:dP3-A4-A2} Suppose that $\Sigma=\{\mathrm{A}_{1},
\mathbb{A}_{2}\}$. Then $\mathrm{lct}(S)=1/2$.
\end{lemma}

\begin{proof}
Let $P$ be a point in $\Sigma$ of type $\mathbb{A}_{1}$. We may
assume that $P\in L_{1}$. Then $r=5$, and~it~easily follows from
\cite{BW79} that the surface $S$ contains lines
$L_{6},L_{7},L_{8},L_{9},L_{10},L_{11}$ such that
$$
P=L_{1}\cap L_{6}\cap L_{7}\cap L_{8},\ L_{9}\cap L_{6}\ne\varnothing,\ L_{9}\cap L_{7}\ne\varnothing, L_{9}\cap L_{6}\ne\varnothing%
$$
and $L_{9}\cap L_{7}\ne\varnothing$, $L_{10}\cap
L_{7}\ne\varnothing$, $L_{10}\cap L_{8}\ne\varnothing$,
$L_{11}\cap L_{6}\ne\varnothing$, $L_{11}\cap
L_{8}\ne\varnothing$. Then
$$
L_{2}\not\ni P\not\in L_{3},\ L_{4}\not\ni P\not\in L_{5},\ L_{6}\not\ni O\not\in L_{7},\ L_{8}\not\ni O\not\in L_{9},\ L_{10}\not\ni O\not\in L_{11},%
$$
which implies that $-K_{S}\sim L_{3}+L_{4}+L_{5}\sim
2L_{1}+L_{2}\sim L_{3}+L_{4}+L_{5}$ and
$$
2L_{1}+L_{2}\sim L_{1}+L_{3}+L_{6}\sim L_{1}+L_{4}+L_{7}\sim L_{1}+L_{5}+L_{8}\sim L_{6}+L_{7}+L_{9}\sim L_{7}+L_{8}+L_{10}\sim L_{6}+L_{8}+L_{11}.%
$$

We see that $\mathrm{lct}(S)\leqslant 1/2$. Therefore, we may
assume that $(S,\lambda D)$ is not log canonical.

Arguing as in the proof of Lemma~\ref{lemma:dP3-A3-A1}, we see
that $\mathrm{LCS}(S,\lambda D)=O$. We may assume that
$$
\bar{L}_{1}\cdot E_{1}=\bar{L}_{2}\cdot E_{1}=\bar{L}_{3}\cdot E_{2}=\bar{L}_{4}\cdot E_{2}=\bar{L}_{5}\cdot E_{2}=1,\ \bar{L}_{1}\cdot E_{2}=\bar{L}_{2}\cdot E_{2}=\bar{L}_{3}\cdot E_{1}=\bar{L}_{4}\cdot E_{1}=\bar{L}_{5}\cdot E_{1}=0.%
$$

It follows from elementary calculations that
$$
\bar{L}_{1}\equiv\alpha^{*}\big(L_{1}\big)-\frac{2}{3}E_{1}-\frac{1}{3}E_{2},\ \bar{L}_{2}\equiv\alpha^{*}\big(L_{2}\big)-\frac{2}{3}E_{1}-\frac{1}{3}E_{2},%
$$
which implies that we may assume that either
$L_{1}\not\subseteq\mathrm{Supp}(D)$ or
$L_{2}\not\subseteq\mathrm{Supp}(D)$. But
$$
\bar{L}_{3}\equiv\alpha^{*}\big(L_{3}\big)-\frac{1}{3}E_{1}-\frac{2}{3}E_{2},\ \bar{L}_{4}\equiv\alpha^{*}\big(L_{4}\big)-\frac{1}{3}E_{1}-\frac{2}{3}E_{2},\ \bar{L}_{5}\equiv\alpha^{*}\big(L_{5}\big)-\frac{1}{3}E_{1}-\frac{2}{3}E_{2},\,%
$$
which easily implies that we may assume that the support of the
divisor $D$ does not contain one of the lines $L_{3}$, $L_{4}$,
$L_{5}$. Arguing as in the proof of Lemma~\ref{lemma:dP3-A2}, we
obtain a contradiction.
\end{proof}

\begin{lemma}
\label{lemma:dP3-A2-A2} Suppose that
$\Sigma=\{\mathbb{A}_{2},\ldots,\mathbb{A}_{2}\}$ and
$|\Sigma|\geqslant 2$. Then $\mathrm{lct}(S)=1/3$.
\end{lemma}

\begin{proof}
Let $P$ be a point in $\Sigma$ such that $P\ne O$. We may assume
that $P\in L_{1}$. Then $-K_{S}\sim 3L_{1}$, which implies that
$\mathrm{lct}(S)\leqslant 1/3$. Thus, we may assume that
$(S,\lambda D)$ is not log canonical.

We may assume that $(S,\lambda D)$ is not log canonical at the
point $O$ by Lemma~\ref{lemma:smooth-points-1-3}. Then
$$
\bar{L}_{1}\equiv\alpha^{*}\big(L_{1}\big)-\frac{1}{3}E_{1}-\frac{2}{3}E_{2},
$$
where we assume that $\bar{L}_{1}\cap E_{2}\ne \varnothing$. Thus,
we may assume that $L_{1}\not\subseteq\mathrm{Supp}(D)$ due to
Remark~\ref{remark:convexity}, which implies that $a_{2}\leqslant
1$, because $\bar{D}\cdot\bar{L}_{1}\geqslant 0$. Arguing as in
the proof of Lemma~\ref{lemma:dP3-A2}, we see~that
$$
3\geqslant a_{1}+a_{2}\leqslant 3,\ 2a_{1}\geqslant a_{2},\ 2a_{2}\geqslant a_{1},\ 1\geqslant a_{2},%
$$
which implies that there is a point $Q\in E_{1}\cup E_{2}$ such
that the log pair
$(\bar{S},\lambda(\bar{D}+a_{1}E_{1}+a_{2}E_{2}))$ is not log
canonical at the point $Q$. Arguing as in the proof of
Lemma~\ref{lemma:dP3-A2}, we see that
$$
\left\{%
\aligned
&Q\in E_{1}\setminus (E_{1}\cap E_{2})\Rightarrow 2a_{1}>a_{2}+3,\\%
&Q\in E_{1}\cap E_{2}\Rightarrow 2a_{1}>3\ \text{and}\ 2a_{2}>3,\\%
&Q\in E_{2}\setminus (E_{2}\cap E_{1})\Rightarrow 2a_{2}>a_{1}+3,\\%
\endaligned\right.%
$$
which easily leads to a contradiction.
\end{proof}

\begin{lemma}
\label{lemma:dP3-A2-A2-A1} Suppose that
$\Sigma=\{\mathbb{A}_{1},\mathbb{A}_{2},\mathbb{A}_{2}\}$. Then
$\mathrm{lct}(S)=1/3$.
\end{lemma}

\begin{proof}
Let $P\ne O$ be a point in $\Sigma$ of type $\mathbb{A}_{2}$. We
may assume that $P\in L_{1}$. Then $-K_{S}\sim 3L_{1}$, which
implies that $\mathrm{lct}(S)\leqslant 1/3$. Thus, we may assume
that $(S,\lambda D)$ is not log canonical.

We may assume that $L_{1}\not\subseteq\mathrm{Supp}(D)$ due to
Remark~\ref{remark:convexity}. But $\mathrm{LCS}(S,\lambda
D)\subseteq\Sigma$ by Lemma~\ref{lemma:smooth-points-1-3}.

Arguing as in the proof of Lemma~\ref{lemma:dP3-A5-A1}, we see
that $\mathrm{LCS}(S,\lambda D)\subseteq O\cup P$, which easily
leads to a contradiction (see the proof of
Lemma~\ref{lemma:dP3-A2-A2}).
\end{proof}

\begin{lemma}
\label{lemma:dP3-A3-A1-A1} Suppose that
$\Sigma=\{\mathbb{A}_{1},\mathbb{A}_{1},\mathbb{A}_{3}\}$. Then
$\mathrm{lct}(S)=1/2$.
\end{lemma}

\begin{proof}
Let $P_{1}$ and $P_{2}$ be points in $\Sigma$ of type
$\mathrm{A}_{1}$. Then we may assume that $P_{1}\in L_{1}$ and
$P_{2}\in L_{2}$, while we have $r=3$. The surface $S$ contains
lines $L_{4}$ and $L_{5}$ such that
$$
P_{1}\in L_{4}\ni P_{2},\ O\not\in L_{4},\ P_{1}\not\in L_{3}\not\ni P_{2},\ L_{5}\cap\Sigma=\varnothing,%
$$
which implies that $L_{5}\cap L_{3}\ne\varnothing$, $L_{5}\cap
L_{4}\ne\varnothing$, $L_{5}\cap L_{1}=\varnothing$, $L_{5}\cap
L_{2}=\varnothing$. Then
\begin{equation}
\label{equation:A3-A1-A1-equivalences}
-K_{S}\sim L_{1}+L_{2}+L_{4}\sim L_{3}+2L_{1}\sim L_{3}+2L_{2}\sim 2L_{3}+L_{5}\sim 2L_{4}+L_{5},%
\end{equation}
which implies that $\mathrm{lct}(S)\leqslant 1/2$. We may assume
that $(S,\lambda D)$ is not log canonical.

Put $D=\mu_{i} L_{i}+\Omega_{i}$, where $\mu_{i}$ is a
non-negative number, and $\Omega_{i}$ is an~effective
$\mathbb{Q}$-divisor, whose support does not contain the line
$L_{i}$. Let us show that $\mu_{i}<1/\lambda$ for every
$i=1,\ldots,5$.

Suppose that $\mu_{1}\geqslant 1/\lambda>2$. It follows from
equivalences~\ref{equation:A3-A1-A1-equivalences} and
Remark~\ref{remark:convexity} that we may assume that
$L_{3}\not\subseteq\mathrm{Supp}(D)$. Therefore, we have
$$
1=L_{3}\cdot D=L_{3}\cdot\big(\mu_{1}L_{1}+\Omega_{1}\big)\geqslant \mu_{1} L_{3}\cdot L_{1}=\mu_{1}/2>1,%
$$
which is a contradiction.  Similarly, we see that
$\mu_{2}<1/\lambda$, $\mu_{3}<1/\lambda$, $\mu_{4}<1/\lambda$,
$\mu_{5}<1/\lambda$.

Arguing as in the proof of Lemma~\ref{lemma:dP3-A1-single}, we see
that $\mathrm{LCS}(S,\lambda D)$ does not contain curves and
smooth points of $S$. It follows from
Remark~\ref{remark:connectedness} that $\mathrm{LCS}(S,\lambda D)$
consist of one point in $\Sigma$.

Suppose that $\mathrm{LCS}(S,\lambda D)=P_{1}$. Let
$\beta\colon\breve{S}\to S$ be a blow up of the point $P_{1}$.
Then
$$
\mu_{4}\breve{L}_{4}+\breve{\Omega}\equiv\beta^{*}\big(\mu_{4}L_{4}+\Omega\big)-\big(\mu_{4}/2+\epsilon\big)G,
$$
where $G$ is the exceptional curve of the birational morphism
$\beta$, $\breve{L}_{4}$ and $\breve{\Omega}$ are proper
transforms of the divisors $L_{4}$ and $\Omega$ on the surface
$\breve{S}$, respectively, and $\epsilon$ is a positive rational
number. Then
$$
0\leqslant\Big(\mu_{4}\breve{L}_{4}+\breve{\Omega}\Big)\breve{H}=\Big(\beta^{*}\big(\mu_{4}L_{4}+\Omega\big)-\big(\mu_{4}/2+\epsilon\big)G\Big)\cdot\Big(\beta^{*}\big(-K_{S}\big)-G\Big)=3-\mu_{4}-2\epsilon,
$$
where $\breve{H}$ is a general curve in $|-K_{\breve{S}}-G|$.
Thus, there is a point $P\in G$ such that the log pair
$$
\Big(\breve{S},\ \mu_{4}\breve{L}_{4}+\breve{\Omega}+\big(\mu_{4}/2+\epsilon\big)G\Big)%
$$
is not log canonical at $P$. We have
$1-\epsilon=\breve{\Omega}\cdot\breve{L}_{4}\geqslant 0$, which
implies that $\epsilon\leqslant 1$. Then
$$
2\epsilon=\breve{\Omega}\cdot G>2
$$
in the case when $P\not\in\breve{L}_{4}$ (see
Remark~\ref{remark:adjunction}). Thus, we see that
$P\in\breve{L}_{4}$. Then
$$
1-\epsilon=\breve{\Omega}\cdot\breve{L}_{4}>2-\mu_{4}/2-\epsilon,%
$$
due to Remark~\ref{remark:adjunction}. Thus, we see that
$\mu_{4}>2$, which is a contradiction.

Similarly, we see that $P_{2}\not\in\mathrm{LCS}(S,\lambda D)$.
Then $\mathrm{LCS}(S,\lambda D)=O$. We may assume that
$$
\bar{L}_{1}\cdot E_{1}=\bar{L}_{2}\cdot E_{3}=\bar{L}_{3}\cdot E_{2}=1,\ \bar{L}_{1}\cdot E_{2}=\bar{L}_{1}\cdot E_{3}=\bar{L}_{2}\cdot E_{1}=\bar{L}_{2}\cdot E_{2}=\bar{L}_{3}\cdot E_{1}=\bar{L}_{3}\cdot E_{3}=0.%
$$

It follows from elementary calculations that
$$
\bar{L}_{1}\equiv\alpha^{*}\big(L_{1}\big)-\frac{3}{4}E_{1}-\frac{1}{2}E_{2}-\frac{1}{4}E_{3},\ \bar{L}_{2}\equiv\alpha^{*}\big(L_{2}\big)-\frac{1}{4}E_{1}-\frac{1}{2}E_{2}-\frac{3}{4}E_{3},\ \bar{L}_{3}\equiv\alpha^{*}\big(L_{3}\big)-\frac{1}{2}E_{1}-E_{2}-\frac{1}{2}E_{3},%
$$
which implies that the singularities of the log pairs
$$
\left(S,\ \frac{1}{2}\Big(2L_{1}+L_{3}\Big)\right)\ \text{and}\ \left(S,\ \frac{1}{2}\Big(2L_{2}+L_{3}\Big)\right)%
$$
are log canonical. But we may assume that either
$L_{1}\not\subseteq\mathrm{Supp}(D)\not\supseteq L_{2}$ or
$L_{3}\not\subseteq\mathrm{Supp}(D)$,~because the
equivalences~\ref{equation:A3-A1-A1-equivalences} hold. Now the
proof of Lemma~\ref{lemma:dP3-A3} leads to a contradiction.
\end{proof}

\begin{lemma}
\label{lemma:dP3-A2-A1-A1} Suppose that
$\Sigma=\{\mathbb{A}_{1},\mathbb{A}_{1},\mathbb{A}_{2}\}$. Then
$\mathrm{lct}(S)=1/2$.
\end{lemma}

\begin{proof}
Let $P_{1}\ne P_{2}$ be points in $\Sigma$ of type
$\mathrm{A}_{1}$. Then we may assume that $P_{1}\in L_{1}$ and
$P_{2}\in L_{4}$, while we have $r=4$. The surface $S$ contains
lines $L_{5},L_{6},L_{7},L_{8}$ such that
$$
P_{1}\in L_{5},\ P_{2}\in L_{6},\ P_{1}\in L_{7}\ni P_{2},\ O\not\in L_{8},\ P_{1}\not\in L_{8}\not\ni P_{2},%
$$
which implies that $L_{8}\cap L_{7}\ne\varnothing$, $L_{8}\cap
L_{2}\ne\varnothing$, $L_{8}\cap L_{3}\ne\varnothing$, $L_{2}\cap
L_{7}=\varnothing$, $L_{3}\cap L_{7}=\varnothing$. Then
$$
L_{1}+L_{4}+L_{7}\sim L_{2}+2L_{1}\sim L_{3}+2L_{4}\sim 2L_{7}+L_{8}\sim L_{2}+L_{3}+L_{8}\sim L_{1}+L_{3}+L_{5}\sim L_{4}+L_{2}+L_{6},%
$$
and $-K_{S}\sim L_{1}+L_{4}+L_{7}$. Then $\mathrm{lct}(S)\leqslant
1/2$. We may assume that $(S,\lambda D)$ is not log~canonical.

Arguing as in the proof of Lemma~\ref{lemma:dP3-A3-A1-A1}, we see
that $\mathrm{LCS}(S,\lambda D)=O$. We may assume that
$$
\bar{L}_{1}\cdot E_{1}=\bar{L}_{2}\cdot E_{1}=\bar{L}_{3}\cdot E_{2}=\bar{L}_{4}\cdot E_{2}=1,\ \bar{L}_{1}\cdot E_{2}=\bar{L}_{2}\cdot E_{2}=\bar{L}_{3}\cdot E_{1}=\bar{L}_{4}\cdot E_{1}=0.%
$$

The log pair $(S,L_{1}+\frac{1}{2}L_{2}))$ is log canonical,
because  the equivalences
$$
\bar{L}_{1}\equiv\alpha^{*}\big(L_{1}\big)-\frac{2}{3}E_{1}-\frac{1}{3}E_{2},\ \bar{L}_{2}\equiv\alpha^{*}\big(L_{2}\big)-\frac{2}{3}E_{1}-\frac{1}{3}E_{2}%
$$
hold. So, we may assume that either
$L_{1}\not\subseteq\mathrm{Supp}(D)$ or
$L_{2}\not\subseteq\mathrm{Supp}(D)$, because $-K_{S}\sim
2L_{1}+L_{2}$.

Similarly, we may assume that either
$L_{3}\not\subseteq\mathrm{Supp}(D)$ or
$L_{4}\not\subseteq\mathrm{Supp}(D)$, which very easily leads to
a~contradiction (see the proof of Lemma~\ref{lemma:dP3-A2}).
\end{proof}

Therefore, it follows from \cite{BW79} that the assertion of
Theorem~\ref{theorem:main} is proved.

\section{Invariant thresholds.}
\label{section:invariants}

In this section we prove the following two lemmas.

\begin{lemma}
\label{lemma:Cayley-cubic} Let $S$ be a cubic surface in
$\mathbb{P}^{3}$ given by the
equation~\ref{equation:Cayley-cubic}. Then
$\mathrm{lct}(S,\mathrm{S}_{4})=1$.
\end{lemma}

\begin{proof}
Let $O_{1},\ldots,O_{4}$ be singular points of the surface
$S_{1}$, and let $L_{ij}$ be a line in $S$ that contains the
points $O_{i}$ and $O_{j}$, where $i\ne j$. Then $\mathrm{S}_{4}$
acts transitively on $\{O_{1},\ldots,O_{4}\}$ and
$\{L_{12},\ldots,L_{34}\}$.

Let $T$ be a curve that is cut out on $S$ by the equation
$x+y+z+t=0$. Then $T$ is $\mathrm{S}_{4}$-in\-va\-ri\-ant, which
implies that $\mathrm{lct}(S,\mathrm{S}_{4})\leqslant 1$. Suppose
that $\mathrm{lct}(S,\mathrm{S}_{4})<1$. Then there
is~an~effective~$\mathrm{S}_{4}$-in\-va\-ri\-ant~$\mathbb{Q}$-divisor
$D$ such that $D\equiv -K_{S}$, and $(S,\lambda D)$ is not log
canonical, where $\lambda\in\mathbb{Q}$~and~$\lambda<1$.

The surface $S$ does not contain $\mathrm{S}_{4}$-invariant
points, because the group $\mathrm{S}_{4}$ does no have faithful
two-dimensional linear representations. Then
$\mathrm{LCS}(S,\lambda D)$ contains a curve by
Remark~\ref{remark:connectedness}.

There are a reduced $\mathrm{S}_{4}$-inva\-ri\-ant curve $C\subset
S$ and a rational number $m\geqslant 1/\lambda$ such that
$$
D=mC+\Omega,
$$
where $\Omega$ is an effective di\-vi\-sor, whose
support does not contain components of $C$. Then
$$
3=-K_{S}\cdot D=m\mathrm{deg}\big(C\big)-K_{S}\cdot\Omega\geqslant m \mathrm{deg}\big(C\big)>\mathrm{deg}\big(C\big),%
$$
which implies that either $C$ is a line, or $C$ is a conic.

Suppose that the curve $C$ is not an irreducible conic. Let $L$ be
any irreducible component the~curve $C$. Then $L$ is a line. Let
$M$ be a general hyperplane section of $S$ that contains $L$.~Then
$$
M=L+\bar{L}\sim -K_{S},
$$
where $\bar{L}$ is an irreducible conic. We have
$$
2=\bar{L}\cdot D=m \bar{L}\cdot L+\bar{L}\cdot\Omega\geqslant m \bar{L}\cdot C>\bar{L}\cdot L,%
$$
which implies that $L\cap\{O_{1},\ldots,O_{4}\}\ne\varnothing$.
Then $L\in\{L_{12},\ldots,L_{34}\}$, which is~impossible, because
the curve $C$ contains at most two components.

We see that $\mathrm{LCS}(S,\lambda D)$ does not contains lines,
and $C$ is an irreducible conic.

Let $R$ be a hyperplane section of the surface $S$ that contains
the conic $C$. Then
$$
R=C+\bar{C}\sim-K_{S},
$$
where $\bar{C}$ is a $\mathrm{S}_{4}$-invariant line. The
intersection $\bar{C}\cap C$ consists of two points.

The log pair $(S,\bar{C}+C)$ is log canonical. We may assume
$\bar{C}\not\subseteq\mathrm{Supp}(\Omega)$ by
Remark~\ref{remark:convexity}.~Then
$$
1=\bar{S}\cdot D=m \bar{C}\cdot C+\bar{C}\cdot\Omega\geqslant m \bar{C}\cdot C>\bar{C}\cdot C,%
$$
which implies that $\bar{C}\cap C\subset \{O_{1},\ldots,O_{4}\}$.
Then $\bar{C}\in\{L_{12},\ldots,L_{34}\}$, which is impossible.
\end{proof}

\begin{lemma}
\label{lemma:cubic-A2-A2-A2} Let $S$ be a cubic surface in
$\mathbb{P}^{3}$ given by the
equation~\ref{equation:cubic-A2-A2-A2}. Then
$\mathrm{lct}(S,\mathrm{S}_{3}\times\mathbb{Z}_{3})=1$.
\end{lemma}

\begin{proof}
Put $G=\mathrm{S}_{3}\times\mathbb{Z}_{3}$. Let
$O_{1},O_{2},O_{3}$ be singular points of the surface $S$, and \
$L_{i}\subset S$~be~a~line such that $O_{i}\not\in L_{i}$. Then
$\mathrm{lct}(S,G)\leqslant 1$, because the curve
$L_{1}+L_{2}+L_{3}$ is $G$-invariant.

We suppose that $\mathrm{lct}(S,G)<1$. Then there
is~an~effective~$G$-invariant $\mathbb{Q}$-divisor~$D$~such~that
the~equivalence $D\equiv -K_{S}$ holds, and $(S,\lambda D)$ is not
log canonical, where $\lambda\in\mathbb{Q}$~and~$\lambda<1$.

The surface $S$ does not contain $G$-invariant points. Then
$|\mathrm{LCS}(S,\lambda D)|=+\infty$ by
Remark~\ref{remark:connectedness}, which implies that there are a
$G$-inva\-ri\-ant curve $C\subset S$ and a  rational number
$m\geqslant 1/\lambda$ such~that
$$
D=mC+\Omega,
$$
where $\Omega$ is an effective $\mathbb{Q}$-divisor, whose support
does not contain components of $C$. Then
$$
3=-K_{S}\cdot D=m \mathrm{deg}\big(C\big)-K_{S}\cdot\Omega\geqslant m\mathrm{deg}\big(C\big)>\mathrm{deg}\big(C\big),%
$$
which implies that either $C$ is a line, or $C$ is a conic.

The only lines contained in $S$ are the lines $L_{1}$, $L_{2}$,
$L_{3}$. The group $G$ acts on the set
$$
\big\{L_{1},L_{2},L_{3}\big\}
$$
transitively. Hence, the curve $C$ is neither a line, nor conic.
\end{proof}

\section{Fiberwise maps.}
\label{section:fibers}

Let $Z$ be a smooth curve. Suppose that there is a commutative
diagram
\begin{equation}
\label{equation:commutative-diagram} \xymatrix{
&V\ar@{->}[d]_{\pi}\ar@{-->}[rr]^{\rho}&&\ \bar{V}\ar@{->}[d]^{\bar{\pi}}&\\%
&Z\ar@{=}[rr]&&Z&}
\end{equation}
such that $\pi$ and $\bar{\pi}$ are flat morphisms, and $\rho$ is
a birational map that induces an isomorphism
\begin{equation}
\label{equation:isomorphism}
\rho\big\vert_{V\setminus X}\colon V\setminus X\longrightarrow\bar{V}\setminus\bar{X},%
\end{equation}
where $X$ and $\bar{X}$ are scheme fibers of $\pi$ and $\bar{\pi}$
over a point $O\in Z$, respectively. Suppose that
\begin{itemize}
\item the varieties $V$ and $\bar{V}$ have terminal $\mathbb{Q}$-factorial singularities,%
\item the divisors $-K_{V}$ and $-K_{\bar{V}}$ are $\pi$-ample and $\bar{\pi}$-ample, respectively,%
\item the fibers $X$ and $\bar{X}$ are irreducible.
\end{itemize}

The following example is due to  \cite{Co96}.

\begin{example}
\label{example:dP3-D4} Suppose that $X$ is a smooth cubic surface
that contains lines $L_{1},L_{2},L_{3}$ such that the intersection
$L_{1}\cap L_{2}\cap L_{3}$ consists of single point $P\in X$.
There is commutative~diagram
$$
\xymatrix{
U\ar@{->}[d]_{\alpha}\ar@{-->}[rr]^{\psi}&&\bar{U}\ar@{->}[d]^{\beta}\\%
V\ar@{-->}[rr]^{\rho}&&\ \bar{V},}
$$ %
where $\alpha$ is a blow up of $P$, $\psi$ is an antiflip in the
proper transforms of $L_{1},L_{2},L_{3}$, and
$\beta$~is~a~contraction of the~proper transform of the fiber $X$.
Then $\bar{X}$ is a cubic surface that has one singular point of
type $\mathbb{D}_{4}$. We have $\mathrm{lct}(X)=2/3$ and
$\mathrm{lct}(\bar{X})=1/3$ (see
Example~\ref{example:smooth-del-Pezzo-surfaces}~and~Lemma~\ref{lemma:dP3-D4}).
\end{example}

Which kind of conditions on the fibers $X$ and $\bar{X}$ imply
that $\rho$~is~biregular?

\begin{example}
\label{example:Jihun-del-Pezzos} Suppose that both fibers $X$ and
$\bar{X}$ are nonsingular del Pezzo surfaces such~that
the~inequality $K_{X}^{2}=K_{\bar{X}}^{2}\leqslant 4$ holds. Then
$\rho$ is an isomorphism (see \cite{Pa01}).
\end{example}

The question we asked is local by the curve $Z$. Thus, in the
following, we will not not~assume that the curve $Z$ is
projective. Let us consider two examples with $Z=\mathbb{C}^{1}$
(see \cite{Pa01}).

\begin{example}
\label{example:dP3-E6} Let $V$ be $\bar{V}$ subvarieties in
$\mathbb{C}^{1}\times\mathbb{P}^{3}$ given by the equations
$$
x^{3}+y^{2}z+z^{2}w+t^{12}w^{3}=0\ \text{and}\ x^{3}+y^{2}z+z^{2}w+w^{3}=0,%
$$
respectively, where $t$ is a coordinate on $\mathbb{C}^{1}$, and
$(x,y,z,w)$ are coordinates on $\mathbb{P}^{3}$. The projections
$$
\pi\colon V\longrightarrow\mathbb{C}^{1}\ \text{and}\ \bar{\pi}\colon \bar{V}\longrightarrow\mathbb{C}^{1}%
$$
are fibrations into cubic surfaces. Let $O$ be the point on
$\mathbb{C}^{1}$ given by $t=0$. Then $\bar{X}$ is smooth, the
surface $X$ has one singular point of type $\mathbb{E}_{6}$. Put
$Z=\mathbb{C}^{1}$. Then the map
$$
\big(x,y,z,w\big)\longrightarrow \big(t^{2}x,t^{3}y,z,t^{6}w\big)%
$$
induces a birational map $\rho\colon V\dasharrow\bar{V}$ such that
the~diagrams~\ref{equation:commutative-diagram} and
isomorphism~\ref{equation:isomorphism} exist, and~$\rho$ is not
biregular. But $\mathrm{lct}(X)=1/6$ and
$\mathrm{lct}(\bar{X})=2/3$ (see
Example~\ref{example:smooth-del-Pezzo-surfaces}~and~Lemma~\ref{lemma:dP3-E6}).
\end{example}

\begin{example}
\label{example:dP3-D5} Let $V$ be $\bar{V}$ subvarieties in
$\mathbb{C}^{1}\times\mathbb{P}^{3}$ given by the equations
$$
wz^{2}+zx^{2}+y^{2}x+t^{8}w^{3}=0\ \text{and}\ wz^{2}+zx^{2}+y^{2}x+w^{3}=0,%
$$
respectively, where $t$ is a coordinate on $\mathbb{C}^{1}$, and
$(x,y,z,w)$ are coordinates on $\mathbb{P}^{3}$. The projections
$$
\pi\colon V\longrightarrow\mathbb{C}^{1}\ \text{and}\ \bar{\pi}\colon \bar{V}\longrightarrow\mathbb{C}^{1}%
$$
are fibrations into cubic surfaces. Let $O$ be the point on
$\mathbb{C}^{1}$ given by $t=0$. Then $\bar{X}$ is smooth, the
surface $X$ has one singular point of type $\mathbb{D}_{5}$. Put
$Z=\mathbb{C}^{1}$. Then the map
$$
\big(x,y,z,w\big)\longrightarrow \big(t^{2}x,ty,z,t^{4}w\big)%
$$
induces a birational map $\rho\colon V\dasharrow\bar{V}$ such that
the~diagrams~\ref{equation:commutative-diagram} and
isomorphism~\ref{equation:isomorphism} exist, and~$\rho$ is not
biregular. But $\mathrm{lct}(X)=1/4$ and
$\mathrm{lct}(\bar{X})=2/3$ (see
Example~\ref{example:smooth-del-Pezzo-surfaces}~and~Lemma~\ref{lemma:dP3-D5}).
\end{example}

The following result holds (see
Examples~\ref{example:smooth-del-Pezzo-surfaces} and
\ref{example:Jihun-del-Pezzos}).

\begin{theorem}
\label{theorem:Park-Cheltsov} The map $\rho$ is an isomorphism if
one of the~following conditions hold:
\begin{itemize}
\item the~varieties $X$ and $\bar{X}$ have log terminal singularities, and $\mathrm{lct}(X)+\mathrm{lct}(\bar{X})>1$;%
\item the~variety $X$ has log terminal singularities, and $\mathrm{lct}(X)\geqslant 1$.%
\end{itemize}
\end{theorem}

\begin{proof}
Suppose that the variety $X$ has log terminal singularities, the
inequality $\mathrm{lct}(X)\geqslant 1$ holds, and $\rho$ is not
an isomorphism. Let $D$ be a general very ample divisor on $Z$.
Put
$$
\Lambda=\big|-nK_{V}+\pi^{*}(nD)\big|,\ \Gamma=\big|-nK_{\bar{V}}+\bar{\pi}^{*}(nD)\big|,\ \bar{\Lambda}=\rho(\Lambda),\ \bar{\Gamma}=\rho^{-1}(\Gamma),%
$$
where $n$ is a natural number such that $\Lambda$ and $\Gamma$
have no base points. Put
$$
M_{V}={\frac{2\varepsilon}{n}}\,\Lambda+{\frac{1-\varepsilon}{n}}\,\overline{\Gamma},\ M_{\bar{V}}={\frac{2\varepsilon}{n}}\,\bar{\Lambda}+{\frac{1-\varepsilon}{n}}\,\Gamma,%
$$
where $\varepsilon$ is a positive rational number.

The log pairs $(V, M_{V})$ and $(\bar{V}, M_{\bar{V}})$ are
birationally equivalent, and $K_{V}+M_{V}$ and
$K_{\bar{V}}+M_{\bar{V}}$~are ample. The~uniqueness of~canonical
model (see Theorem~1.3.20 in \cite{Ch05umn}) implies that
$\rho$~is~bire\-gular if the singularities of both log pairs $(V,
M_{V})$ and $(V, M_{\bar{V}})$ are canonical.

The linear system $\Gamma$ does not have base points. Thus, there
is a rational number $\varepsilon$ such that the log pair
$(\bar{V}, M_{\bar{V}})$ is canonical. So, the log pair $(V,
M_{V})$ is not canonical. Then the~log pair
$$
\Big(V,\ X+{\frac{1-\varepsilon}{n}}\,\bar{\Gamma}\Big)%
$$
is not log canonical, because $\Lambda$ does not have not base
points, and $\bar{\Gamma}$ does not have base points outside of
the~fiber $X$, which is a Cartier divisor on the~variety $V$. The
log pair
$$
\Big(X,\ {\frac{1-\varepsilon}{n}}\,\bar{\Gamma}\big\vert_{X}\Big)
$$
is not log canonical by Theorem~17.6 in \cite{Ko91}, which is
impossible, because $\mathrm{lct}(X)\geqslant 1$.

To conclude the proof we may assume that the varieties $X$ and
$\bar{X}$ have log terminal singularities, the inequality
$\mathrm{lct}(X)+\mathrm{lct}(\bar{X})>1$ holds, and  $\rho$ is
not an isomorphism.

Let $\Lambda$, $\Gamma$, $\bar{\Lambda}$, $\bar{\Gamma}$ and $n$
be the~same as in the~previous case. Put
$$
M_{V}=\frac{\mathrm{lct}(\bar{X})-\varepsilon}{n}\,\Lambda+\frac{\mathrm{lct}(X)-\varepsilon}{n}\,\overline{\Gamma},\ M_{\bar{V}}=\frac{\mathrm{lct}(\bar{X})-\varepsilon}{n}\,\bar{\Lambda}+\frac{\mathrm{lct}(X)-\varepsilon}{n}\,\Gamma,%
$$
where $\varepsilon$ is a sufficiently  small positive rational
number. Then it follows from the~uniqueness of canonical model
that $\rho$ is biregular if both log pair $(V, M_{V})$ and $(V,
M_{\bar{V}})$ are canonical.

Without loss of generality, we may assume that the~singularities
of the~log pair $(V, M_{V})$ are not canonical. Arguing as in
the~previous case, we see that the~log pair
$$
\Big(X,\
\frac{\mathrm{lct}(X)-\varepsilon}{n}\,\bar{\Gamma}\big\vert_{X}\Big)
$$
is not log canonical, which is impossible, because
$\bar{\Gamma}\vert_{X}\equiv -nK_{X}$.
\end{proof}

The assertion of Theorem~\ref{theorem:Park-Cheltsov} can not
be~improved (see Examples~\ref{example:dP3-D4},
\ref{example:dP3-E6}, \ref{example:dP3-D5}).

\end{document}